\documentclass[11pt]{article}

\usepackage[a4paper,margin=28mm]{geometry}
\usepackage{amsmath,amssymb,amsthm,mathtools}
\usepackage{enumitem}
\usepackage{microtype}
\usepackage[hidelinks]{hyperref}
\usepackage{xcolor}

\newtheorem{theorem}{Theorem}[section]
\newtheorem{proposition}[theorem]{Proposition}
\newtheorem{lemma}[theorem]{Lemma}
\newtheorem{corollary}[theorem]{Corollary}
\theoremstyle{definition}

\numberwithin{equation}{section}
\newcommand{\R}{\mathbb R}
\newcommand{\Z}{\mathbb Z}
\newcommand{\C}{\mathbb C}

\newcommand{\cC}{\mathcal C}

\newcommand{\cF}{\mathcal F}
\newcommand{\cG}{\mathcal G}

\newcommand{\cL}{\mathcal L}
\newcommand{\cM}{\mathcal M}

\newcommand{\cS}{\mathcal S}
\newcommand{\norm}[1]{\left\lVert#1\right\rVert}

\newcommand{\ip}[2]{\left\langle#1,#2\right\rangle}
\newcommand{\dd}{\,\mathrm d}
\newcommand{\Rea}{\operatorname{Re}}
\newcommand{\Ima}{\operatorname{Im}}
\newcommand{\supp}{\operatorname{supp}}
\newcommand{\uloc}{\mathrm{uloc}}
\newcommand{\one}{\mathbf 1}

\title{Global well-posedness of defocusing cubic NLS in $M^{\infty,1}(\R)$}
\author{Friedrich Klaus}
\date{\today}

\begin{document}
\maketitle

\begin{abstract}
We prove global well-posedness of the one-dimensional defocusing cubic
nonlinear Schr\"odinger equation in the modulation space $M^{\infty,1}(\R)$.
This space imposes no spatial decay and contains $C_b^2(\R)$ as well as all
absolutely convergent sums of plane waves.  The result applies to arbitrary
data in this space, including large smooth quasiperiodic profiles and their
localized perturbations.  The proof constructs a nonnegative density satisfying a local conservation law from forward Weyl ratios, which are defined through half-line square-integrable solutions of the associated spectral problem. A suitable nonlinear combination of localized integrals of this density controls the modulation norm. Finally, choosing the spatial localization scale and NLS scaling in a coordinated way makes the accumulated boundary flux small enough to continue every mild solution globally.
\end{abstract}

\tableofcontents
\clearpage

\section{Introduction}\label{sec:main}

We study the Cauchy problem for the defocusing cubic nonlinear
Schr\"odinger equation (NLS) on the real line,
\begin{equation}\label{eq:nls}
 iq_t=-q_{xx}+2|q|^2q,\qquad q(0)=q_0.
\end{equation}
For more than fifty years, this equation has been a central example in the
study of nonlinear dispersive waves and completely integrable systems.
The last decade has brought substantial progress in its low-regularity
theory.  Koch--Tataru \cite{KT2018} and Killip--Vi\c{s}an--Zhang
\cite{KVZ2018} constructed conserved quantities controlling rough Sobolev
norms.  Building on this progress, Harrop-Griffiths--Killip--Vi\c{s}an
\cite{HGV2024} proved global well-posedness in $H^s(\R)$ for every
$s>-1/2$, reaching the sharp threshold in this scale.  These results
exploit complete integrability to go beyond the reach of perturbative
well-posedness arguments.

One manifestation of integrability is an infinite hierarchy of conserved
Hamiltonians whose flows commute.  Its first members include the mass,
momentum, and energy, which for sufficiently regular, decaying functions
are given by
\begin{equation}\label{eq:classical-conserved}
 \begin{aligned}
 \mathsf M(q)&=\int_\R |q|^2\dd x,
 &\mathsf P(q)&=\Ima\int_\R\bar q q_x\dd x,\\
 \mathsf E(q)&=\int_\R\bigl(|q_x|^2+|q|^4\bigr)\dd x.
 \end{aligned}
\end{equation}
Their Hamiltonian flows generate phase rotations, spatial translations,
and NLS evolution, respectively; see \cite{FaddeevTakhtajan1987}.
The integrable structure, discovered by Zakharov and Shabat
\cite{ZakharovShabat1972,ZakharovShabat1973}, supplies conservation laws
beyond these classical quantities and is the starting point for the
global estimates in this paper.

Here we address a different obstruction to global well-posedness: the
absence of spatial decay.  The integrals in \eqref{eq:classical-conserved}
require spatial integrability of their densities and are generally not
finite for non-decaying data, even when the data are smooth.  Already a
nonzero constant has infinite mass and energy.  Periodic data can be
treated on the torus, but a sum of two periodic functions with
incommensurable periods falls outside that setting.  For example, a
global theory on the line should accommodate
\[
 q_0(x)=A\cos x+B\cos(\alpha x),
 \qquad A,B\in\C,\quad \alpha\in\R\setminus\mathbb Q,
\]
as well as localized perturbations of this profile, without a smallness
condition on $A$ or $B$.  Such extended wave patterns lead to a natural
question: can integrability still prevent finite-time breakdown when
the usual conserved integrals are infinite?

Non-decaying backgrounds have long been important in NLS.  The finite-gap
method produces explicit periodic and quasiperiodic solutions
\cite{BelokolosEtAl1994,ItsKotlyarov1976,Kotlyarov1976}, while Zhidkov spaces and
the theory of nonzero boundary conditions provide PDE frameworks for
bounded profiles and finite-energy perturbations
\cite{Gallo2008,Gerard2006,Zhidkov2001}.  More recent work has developed
the theory for slowly decaying data in modulation spaces
\cite{CHKP2017,Klaus2023,OW2020}, for perturbations of periodic backgrounds
\cite{CHKP2021,KlausKunstmann2022}, and for almost periodic data
\cite{FLLZ2025,Oh2015LP,Oh2015AP}.  These results treat different aspects
of the problem, often using spatial integrability, a prescribed background,
or assumptions on the frequencies and spectral data.

The modulation space $M^{\infty,1}(\R)$ offers a common setting for a broad
class of these non-decaying profiles.  Its norm sums the $L^\infty$ norms
of the pieces of a function in frequency intervals of fixed length.
It imposes no spatial integrability and contains $C_b^2(\R)$, the space
of twice continuously differentiable functions with bounded derivatives
up to order two, as well as every absolutely convergent sum of plane waves.
It therefore accommodates smooth quasiperiodic backgrounds, localized
perturbations, and bounded profiles with no recurrence in space.
Note though that the frequency summability in its norm is a restriction:
boundedness alone, or even Bohr almost periodicity, does not imply
membership in $M^{\infty,1}$.
The space is a Banach algebra on which the Schr\"odinger group is strongly
continuous, so local well-posedness follows by Picard iteration.
The difficulty is to obtain a bound that continues these solutions:
localizing the conserved integrals introduces boundary fluxes whose
accumulated effect is not controlled by the local theory.

In \cite{Klaus2023} the author obtained global results for arbitrarily large finite
spatial integrability exponents, assuming enough regularity, and left the
non-decaying endpoint open.  The present paper resolves this endpoint
problem for \eqref{eq:nls}: every initial value in $M^{\infty,1}(\R)$
generates a unique global mild solution, with locally Lipschitz dependence
on the initial data on compact time intervals.  No smallness, periodicity,
arithmetic, or spectral assumption is imposed on data in this space.
This gives a global well-posedness framework for non-decaying profiles
with the frequency summability specified by the modulation norm.

\subsection*{Main result}
\addcontentsline{toc}{subsection}{Main result}

We first define what we mean by modulation spaces. Choose a nonnegative, real, even function
$\varphi\in C_c^\infty((-2,2))$ such that
$\sum_{j\in\Z}\varphi(\xi-j)=1$, and define the frequency projections
\[
 \square_j f=\mathcal F^{-1}
   \bigl(\varphi(\cdot-j)\widehat f\bigr).
\]
The modulation space we are interested in is
\begin{equation}\label{eq:Mdef}
 M^{\infty,1}(\R)
 =\left\{f\in\mathcal S'(\R):
   \norm f_{M^{\infty,1}}
   :=\sum_{j\in\Z}\norm{\square_jf}_{L^\infty}<\infty\right\}.
\end{equation}
This norm makes $M^{\infty,1}$ a Banach space, and different frequency
partitions of the above type give equivalent norms; see
\cite{BGOR2007,Klaus2023}.  A mild solution on
an interval $I$ containing $0$ is a function $q\in C(I;M^{\infty,1})$ which satisfies
\begin{equation}\label{eq:duhamel}
 q(t)=e^{it\partial_x^2}q_0
 -2i\int_0^te^{i(t-s)\partial_x^2}\bigl(|q|^2q\bigr)(s)\dd s.
\end{equation}

\begin{theorem}[Global well-posedness]\label{thm:main}
For every $q_0\in M^{\infty,1}(\R)$ there is a unique global mild solution
\[
 q\in C\bigl(\R;M^{\infty,1}(\R)\bigr)
\]
of \eqref{eq:nls}.  For every $T>0$, its solution map
\[
 M^{\infty,1}\longrightarrow C([-T,T];M^{\infty,1})
\]
is locally Lipschitz.
\end{theorem}

\subsection*{Previous results and almost periodic initial data}
\addcontentsline{toc}{subsection}{Previous results and almost periodic initial data}

The global theory for slowly decaying data has developed along several
directions.  Dodson--Soffer--Spencer \cite{DSS2020} considered bounded
periodic, quasiperiodic, and random profiles and proved local existence for
analytic data on the line.  Their global continuum result in that paper
concerns a regularized nonlinearity.  For the cubic equation itself, they
subsequently obtained global well-posedness for sufficiently smooth data
in $L^p(\R)$ with $2<p<\infty$ \cite{DSS2021}.  Schippa
\cite{Schippa2022} used decoupling and smoothing estimates to improve the
regularity requirements in modulation spaces.  The results of
\cite{Klaus2023} include global well-posedness in $M^{1}_{p,1}$ for
$2\le p<\infty$.  These results allow infinite mass, but their finite
spatial integrability exponent excludes nonzero almost periodic functions.
In joint work with Kunstmann \cite{KlausKunstmann2022}, we treated another
class of non-decaying data, namely an $H^1(\R)$ perturbation of a periodic
profile in $H^s(\mathbb T)$, $s>3/2$.  Theorem~\ref{thm:main} allows
the background itself to be an arbitrary element of $M^{\infty,1}$.

In both research directions, modulation spaces and spaces of sums of periodic and non-periodic Sobolev functions, several results were obtained.
First, the global modulation-space result of Chaichenets--Hundertmark--
Kunstmann--Pattakos \cite{CHKP2017} treats $M^{p,p'}$ with $p>2$
sufficiently close to $2$, while Oh--Wang \cite{OW2020} use
perturbation-determinant conservation laws and Galilean transformations
to obtain global well-posedness in $M^{2,p}$, $2\le p<\infty$.
Neither result reaches the spatial $L^\infty$ endpoint considered here.
Second, the local theory of \cite{CHKP2021} allows an $L^2(\R)$
perturbation of an $H^1(\mathbb T)$ profile for nonlinearities including
the cubic one; its global theorem concerns the quadratic nonlinearity.
The global cubic result in \cite{KlausKunstmann2022} requires the stronger
perturbation and background regularity stated above.

Concerning the classical theory of finite-gap solutions, Theorem~\ref{thm:main} also places regular finite-gap NLS solutions in an
ambient Banach space carrying a well-posed flow.  The initial data need
not be finite-gap or close to a finite-gap potential.  Similarly, the
theorem does not require a finite-energy perturbation of a constant
background: nonconstant smooth quasiperiodic functions generally have
derivatives outside $L^2(\R)$, as do many other elements of
$M^{\infty,1}$.

For almost periodic data there is a separate line of research.  Here
almost periodic means a uniform limit on $\R$ of finite sums of the
functions $e^{i\xi x}$, $\xi\in\R$.  Quasiperiodic functions are obtained
by restricting a continuous periodic function of finitely many variables
to a line.  The corresponding frequencies need not lie in a discrete
subset of $\R$.  Oh \cite{Oh2015AP} proved local well-posedness in
algebras of almost periodic functions with absolutely summable Fourier
coefficients.  He also obtained global solutions for a class of limit
periodic data, under regularity and quantitative periodic approximation
assumptions \cite{Oh2015LP}.  More recently, Papenburg
\cite{Papenburg2025} gave local well-posedness results for several
dispersive equations with quasiperiodic data, including NLS.
Damanik--Li--Xu \cite{DLX2024} treated the NLS Cauchy problem with
polynomial decay of the quasiperiodic Fourier coefficients and studied
its weakly nonlinear dynamics.  The latter direction has also been
developed in higher dimensions by Xu \cite{Xu2025}.  The conditions in
these local and weakly nonlinear results concern decay of coefficients
indexed by the underlying frequency lattice; they should be distinguished
from decay of the initial function in physical space.

Another approach uses dispersive estimates adapted to quasiperiodic
functions.  In \cite{Klaus2023Strichartz} the author proved a space-time averaged
$L^4$ Strichartz estimate for frequencies in $\Z+\sqrt{2}\Z$.  The
averaging in time prevents a direct application of the usual $TT^*$
argument to the Duhamel integral.  Schippa \cite{Schippa2025QP}
obtained finite-time estimates using square function and decoupling
arguments, as well as generalizations of the time-averaged estimate.
His multilinear estimates give local well-posedness for the cubic NLS
with $\nu$ rationally independent frequencies at Sobolev regularity
$s>(\nu-1)/2$ in the frequency-lattice index.  Below this threshold
the solution map fails to be $C^3$ at the origin.  These spaces use
the spatial mean of $|f|^2$, in the sense of Besicovitch almost
periodicity, and need not consist of bounded functions.  Their low
regularity theory is therefore complementary to the global theory
in $M^{\infty,1}$ considered here.

Very recently, Schippa \cite{Schippa2026QP} used short-time bilinear
estimates to obtain local well-posedness results for quasiperiodic
KdV and Benjamin--Ono data while preserving the initial Sobolev
regularity.  Inami \cite{Inami2026} proved long-time Strichartz
estimates for two-frequency Schr\"odinger data with algebraic
frequency ratio, and an endpoint $L^4$ estimate.  These results
illustrate both the progress in dispersive methods for non-decaying
data and the role of arithmetic in quantitative long-time estimates.

A particularly influential question in this context is Deift's conjecture
for KdV: does almost periodicity of the initial data lead to global
solutions which are almost periodic in time?  We refer to
\cite[Problem~1]{Deift2017} for its formulation and motivation.
Damanik--Goldstein \cite{DG2016} proved global existence for small
analytic quasiperiodic KdV data with Diophantine frequencies.
Binder--Damanik--Goldstein--Luki\'c \cite{BDGL2018} established almost
periodicity in time under suitable spectral assumptions, in particular
for this small quasiperiodic class.  The unrestricted strong formulation,
requiring almost periodicity in both space and time, is false:
Chapouto--Killip--Vi\c{s}an \cite{CKV2024} constructed almost periodic
initial data whose bounded KdV evolution loses spatial almost periodicity
at a later time.  This makes the choice of topology and of the admissible
class of initial data essential.

For the defocusing cubic NLS, Fillman--Li--Luki\'c--Zhou
\cite{FLLZ2025} recently proved global existence and almost periodicity
in both space and time for small analytic quasiperiodic data with
Diophantine frequencies.  More generally, their inverse spectral result
applies to almost periodic data whose Dirac operator has purely absolutely
continuous spectrum satisfying quantitative thickness conditions.
Their small-data theorem applies, for example, to
\[
 q_0(x)=\varepsilon\bigl(\cos x+\cos(\sqrt{2}\,x)\bigr)
\]
when $|\varepsilon|$ is sufficiently small.  Theorem~\ref{thm:main}
gives a global mild solution for every $\varepsilon\in\R$, and more
generally for $A\cos x+B\cos(\alpha x)$ with arbitrary $A,B\in\C$
and $\alpha\in\R$.  Thus it answers, in the defocusing case, the
concrete continuation question raised in the introduction of
\cite{Klaus2023}.  Its conclusion concerns global well-posedness in
$M^{\infty,1}$; almost periodicity in time is a further dynamical
property which is not established here.

We explain precisely one consequence for spatial almost periodicity.
For an additive subgroup $\Gamma\subset\R$, let
\[
 M_\Gamma=\overline{\operatorname{span}
 \{e^{i\xi x}:\xi\in\Gamma\}}^{\,M^{\infty,1}}.
\]
This is a closed subalgebra, invariant under conjugation and the
Schr\"odinger group.  The local Picard iteration therefore leaves
$M_\Gamma$ invariant, and Theorem~\ref{thm:main} continues its solutions
globally in that space.  Since $M^{\infty,1}$ embeds into $L^\infty$,
each element of $M_\Gamma$ is spatially almost periodic.  In particular,
if
\[
 q_0(x)=\sum_{\nu\ge1}c_\nu e^{i\xi_\nu x},
 \qquad \sum_{\nu\ge1}|c_\nu|<\infty,
\]
then $q_0\in M_\Gamma$ for the group generated by the $\xi_\nu$:
indeed, $\|e^{i\xi x}\|_{M^{\infty,1}}=1$ for the nonnegative partition
in \eqref{eq:Mdef}.  This gives a global spatially almost periodic
solution without any arithmetic condition on these frequencies.
The assertion is in the $M^{\infty,1}$ topology and does not assert
preservation of the stronger absolute Fourier summability norm.
Thus Theorem~\ref{thm:main} gives a positive answer to the global-existence
and spatial-persistence questions underlying the NLS analogue of Deift's
conjecture for the class $M_\Gamma$, without smallness or arithmetic
assumptions.  Almost periodicity in time is not established here.

Note that the restriction to $M_\Gamma$ is essential to the scope of this conclusion.
Bohr almost periodicity requires uniform approximation by trigonometric
polynomials, whereas the definition of $M_\Gamma$ requires approximation
in the stronger modulation norm.  In fact, the almost periodic KdV data
constructed in \cite{CKV2024} lie outside $M^{\infty,1}$.  To see this,
their construction and nonlinear smoothing theorem
\cite[Theorem~4.3 and Section~5]{CKV2024} give such a datum in the form
\[
 u_* = w(t_*)+r_*,\qquad
 w(t)=e^{-t\partial_x^3}
 \bigl[\operatorname{sq}(\alpha_1x)+\operatorname{sq}(\alpha_2x)\bigr],
 \qquad \operatorname{sq}(x)=\operatorname{sgn}(\sin x),
\]
where $\alpha_1/\alpha_2$ is irrational, $w(t_*)$ is continuous and
almost periodic, and $r_*$ has absolutely summable Fourier coefficients;
in particular, $r_*\in M^{\infty,1}$.  The Fourier frequencies of $w(t_*)$
belong to two arithmetic lattices, with at most $N<\infty$ frequencies in
each interval $(j-2,j+2)$.  Writing $c_\xi$ for its Bohr Fourier
coefficients, the bound of each Fourier coefficient by the uniform norm
and the nonnegative partition in \eqref{eq:Mdef} yield
\[
 \sum_\xi |c_\xi|
 =\sum_{j\in\Z}\sum_\xi\varphi(\xi-j)|c_\xi|
 \le N\sum_{j\in\Z}\|\square_jw(t_*)\|_{L^\infty}.
\]
The left side diverges: the Airy evolution preserves the coefficient
magnitudes $2/(\pi|k|)$ at frequencies $\alpha_\ell k$, for odd
$k\in\Z$ and $\ell=1,2$.  Thus $w(t_*)\notin M^{\infty,1}$ and hence
$u_*\notin M^{\infty,1}$.  This illustrates why the present result is a 
statement for a specified class of almost periodic data, rather than 
for all Bohr almost periodic or merely bounded data.

\paragraph{The focusing equation.}
We only prove our result for the defocusing equation.
The restriction to the defocusing sign enters the global a priori estimate;
the local theory in Section~\ref{sec:local} works for either sign.
For the defocusing equation, the Zakharov--Shabat operator in
\eqref{eq:Lax} is a positive spectral parameter plus a skew-adjoint
operator.  Its resolvent exists at every positive height, independently
of the amplitude, and the forward Weyl ratio satisfies $|g_{\lambda,+}|<1$.
These facts yield the positive density and the coercivity used below.
For the focusing sign the off-diagonal symmetry changes, and neither
this amplitude-independent resolvent bound nor the disk bound is
available for arbitrary bounded data.  Thus the present argument does
not establish the focusing analogue of Theorem~\ref{thm:main}.
This is a limitation of the proof, not an assertion of finite-time
blow-up: in particular, the Sobolev-space theorem of \cite{HGV2024}
holds for both signs.

\paragraph{Recent related work.}
Very recently, Kotani, Wang, Xu, and Zhang \cite{KWXZ2026} proved global
well-posedness for both the focusing and defocusing cubic NLS in
$M^{\infty,1}_5(\R)$, which requires five derivatives in the
modulation-space norm.  Our result treats the defocusing equation in
$M^{\infty,1}(\R)$ without any differentiability assumption on the initial
datum.  Their work has a broader scope in other directions: it constructs
global spectral actions for both NLS hierarchies and establishes, among
other consequences, preservation of spatial almost periodicity and
spectral invariants.  Although both approaches use the Dirac spectral
problem associated with NLS, the proofs follow very different lines:
\cite{KWXZ2026} constructs the flow through the Sato--Segal--Wilson
framework and derives uniform bounds from spectral reconstruction and
compactness, whereas our proof obtains global continuation from a
positive microscopic density, its local flux identity, and coercivity in
$M^{\infty,1}$. The two works thus provide complementary advances in the global theory for non-decaying data, differing both in the scope and regularity of their results and in the methods used to obtain them.

\subsection*{Proof strategy}
\addcontentsline{toc}{subsection}{Proof strategy}

The proof uses integrability to construct a nonnegative density whose integrals over spatial windows remain finite for bounded, non-decaying potentials.  Its conservation
law has a current containing no derivatives of the potential.  We then control the modulation norm using a family of nonlinear spectral components (we call them "spectral features") indexed by integer shifts of the spectral parameter, and estimate the change of the density in large spatial windows. To be precise, the spectral features are the sequence-valued functions
\[
F_n^+[q]=h_\kappa(A_q^++n)(\bar q e_1),\qquad n\in\mathbb Z.
\]
where $h_\kappa$ is the function and $A_q^+$ the self-adjoint operator on $L^2(\mathbb R;\ell^2(\mathbb N))$ defined below.
The main issue is to
make this estimate compatible with the size of the initial windowed
quantity.  A saturation of the density, together with the NLS scaling,
provides the required balance.

We make this more precise. Fix a spectral height $\kappa>0$.  For general non-decaying functions with no prescribed limit behaviour at $\pm\infty$, there is no easy way to define the classical Jost solutions for the Zakharov--Shabat operator. What can still be defined is the so-called forward Weyl ratio
$g_{\lambda,+}$ of the Zakharov--Shabat operator. The forward Weyl ratio is related to the diagonal Green functions $g_{21}$ of
\cite{HGV2024} by
\[
g_{\lambda,+}(x)
=\frac{2g_{21}(x;\lambda)}{2+\gamma(x;\lambda)}.
\]
For sufficiently decaying potentials, it is equivalently the ratio
of the second and first components of the left Jost solution. $g_{\lambda,+}$ satisfies the Ricatti equation
\[
 (g_{\lambda,+})_x=\bar q-2\lambda g_{\lambda,+}-qg_{\lambda,+}^2,
 \qquad |g_{\lambda,+}|\le
 \frac{\|q\|_\infty}{\sqrt{\lambda^2+\|q\|_\infty^2}+\lambda}<1.
\]
At the heights $\lambda_a=a\kappa$, use $d=(1,-2,1)$ and, for integers
$m\ge0$, set
\[
 S_m=\sum_{a=1}^3d_ag_{\lambda_a,+}^m,
 \qquad T_m=\sum_{a=1}^3d_a\lambda_ag_{\lambda_a,+}^m.
\]
Here $m$ indexes ordinary scalar powers of the Weyl ratio.  The two cancellations
$\sum_ad_a=\sum_ad_a\lambda_a=0$ produce the positive density and current
\begin{align*}
 E_\kappa&=30\kappa\sum_{m\ge1}|S_m|^2,\\
 J_\kappa&=60\kappa\sum_{m\ge1}
   \Ima[(2T_m+qS_{m+1})\overline{S_m}],
 \qquad \partial_tE_\kappa+\partial_xJ_\kappa=0.
\end{align*}
We give a quick motivation for these objects. Essentially we want to replace a conserved density by a nonnegative one
without changing its integral for decaying potentials.  At a single
spectral height, the Riccati equation for \(g=g_{\lambda,+}\) gives
\[
 \operatorname{Re}(qg)
 =-\frac12\partial_x\log(1-|g|^2)
   +2\lambda\frac{|g|^2}{1-|g|^2}.
\]
Thus the density \(\operatorname{Re}(qg)\), whose integral is the
logarithmic transmission quantity for sufficiently decaying \(q\),
differs from a nonnegative density by a spatial derivative.

The three-height construction follows the same principle, while the
cancellations in \(d=(1,-2,1)\) provide additional spectral decay.
Writing \(g_a=g_{\lambda_a,+}\), the corresponding identity is
\[
 \frac12\operatorname{Re}\bigl(q(5g_1-4g_2+g_3)\bigr)
 =\partial_xH_\kappa^++E_\kappa,
\]
where
\[
 H_\kappa^+
 =-\frac14\sum_{a,b=1}^3
   \frac{60d_ad_b}{a+b}\log|1-g_a\overline{g_b}|.
\]
The coefficients \((5,-4,1)\) are the row sums of the matrix
\(\bigl(60d_ad_b/(a+b)\bigr)_{a,b=1}^3\).  Pairing the Riccati
equations at different heights and expanding the resulting geometric
series produces the sum of squares defining \(E_\kappa\).
For sufficiently decaying potentials, the spatial derivative integrates
to zero, so \(E_\kappa\) has the same integral as the signed combination
of transmission densities on the left.

Subtracting \(\partial_xH_\kappa^+\) from the density requires adding
\(\partial_tH_\kappa^+\) to its current to preserve the conservation law.
In this corrected current, all derivatives of \(q\) cancel, yielding
the expression for \(J_\kappa\) above.  The resulting density and current
remain well-defined for bounded, non-decaying potentials, even when
their integrals over the whole line are unavailable. Section~\ref{sec:density}
proves the conservation law first for smooth solutions and then on the existing lifespan
of every mild endpoint solution.

For $q_n=\cG_nq=e^{-i(nx+n^2t)}q(t,x+2nt)$ and the unnormalized window
$\chi_{R,y}(x)=\chi((x-y)/R)$, define
\[
 B_{n,R}(t,y)=\int\chi_{R,y}^2E_\kappa[q_n]\dd x,
 \qquad
 k_R(q)=\sum_n\sup_y\sqrt{\frac{B_{n,R}(t,y)}{\kappa+B_{n,R}(t,y)}}.
\]
Taking the supremum over all window centers removes the spatial translation in the Galilean transform, so these suprema depend only on the snapshot $q(t)$, by which we mean its spatial profile $q(t,\cdot)$.  If
$\beta_{n,R}^2=(30\kappa)^{-1}\sup_y B_{n,R}(t,y)$, then
$k_R\asymp\sum_n\min(\beta_{n,R},1)$; and this is what we mean by saturation.
The square root gives the linear size of a small feature, as required
for an $\ell^1$ modulation norm, while saturation bounds each large
feature's contribution by one.  This prevents the initial value sum $k_R(q_0)$ from
growing too quickly with the window size.  Coercivity is retained because
a large feature forces many neighboring spectral features to be large;
Section~\ref{sec:large-window-spectral} makes this precise.

The nonlinear features have an exact self-adjoint realization.  On
$L_x^2\ell_m^2$, with $X_0=0$ and $D=-i\partial_x$, set
\[
 (A_q^+X)_m=\frac DmX_m+i\bar qX_{m-1}-iqX_{m+1},
 \qquad
 h_\kappa(s)=\frac{8\kappa^2}
 {(2\kappa+is)(4\kappa+is)(6\kappa+is)}.
\]
After undoing the Galilean gauge, $S = (S_m)_{m \ge1}$ is
$F_n^+=h_\kappa(A_q^++n)(\bar q e_1)$.  Choose a smooth compactly
supported partition $\sum_n\theta(s+n)=1$ and put
$r_n(s)=\theta(s+n)/h_\kappa(s+n)$.  Section~\ref{sec:fock} proves the
uniformly local reconstruction
\[
 \bar q e_1=\sum_n r_n(A_q^+)F_n^+.
\]
A weighted half-line energy estimate for the resolvent controls the first
component of each reconstructed term, independently of $\|q\|_\infty$.
Section~\ref{sec:observation-coercivity} combines this estimate with the
Schur bound and a frequency-box argument to obtain
\begin{equation}\label{eq:strategy-coercivity}
 \|q\|_\infty\lesssim k_R+k_R^2,
 \qquad \|q\|_{M^{\infty,1}}\lesssim k_R(1+k_R)^9,
 \qquad R\ge1.
\end{equation}
The exponent $9$ is inessential; any fixed polynomial coercivity bound
would suffice for continuation.

Spreading of the adapted features and a summable convolution estimate for
the companion $T$ yield the large-window differential inequality
\begin{equation}\label{eq:strategy-growth}
 D_t^+k_R(q(t))\lesssim R^{-1}k_R(q(t))(1+k_R(q(t)))^2.
\end{equation}
Section~\ref{sec:active-flux} justifies differentiation of the nonattained
spatial suprema and the infinite frequency sum directly for mild solutions.
The factor $R^{-1}$ comes from differentiation of the window.

For $S_\lambda q(x)=\lambda q(\lambda x)$ and $R=H/\lambda^2$, the
initial estimate of Section~\ref{sec:frequency-envelope} is
\begin{equation}\label{eq:strategy-initial}
 k_{H/\lambda^2}(S_\lambda q_0)
 \lesssim1+H^{1/6}\|q_0\|_M^{1/3}
\end{equation}
for sufficiently small $\lambda>0$.  Its proof uses
$\|S_\lambda q_0\|_M\lesssim\lambda\|q_0\|_M$ and the frequency collapse
\[
 \sum_j a_j\langle\,\cdot-\lambda j\rangle^{-3}
 \longrightarrow\Big(\sum_j a_j\Big)\langle\,\cdot\rangle^{-3}
 \quad\hbox{in }\ell^1(\Z),\qquad a_j=\|\square_jq_0\|_\infty.
\]
More explicitly, writing $N=\|q_0\|_M$, the leading envelope is
$\sqrt H\,N\langle n\rangle^{-3}$.  Its unsaturated sum is of size
$\sqrt H\,N$, whereas
\[
 \sum_n\min\{\sqrt H\,N\langle n\rangle^{-3},1\}
 \lesssim 1+H^{1/6}N^{1/3}.
\]
This reduction is what allows the window gain to dominate the quadratic
relative growth in \eqref{eq:strategy-growth}; the unsaturated initial
bound would cancel that gain in the same bootstrap.
To reach original time $T$, the scaled solution must reach
$\tau=T/\lambda^2$, so $\tau/R=T/H$.  Choosing
$H=C_*(1+T)^{3/2}(1+\|q_0\|_M)$ closes the bootstrap.
Section~\ref{sec:endpoint} applies coercivity and the local continuation
criterion directly to the maximal mild solution; local stability gives the
asserted dependence on the initial datum.

\subsection*{Acknowledgements}
This paper was written with substantial help of GPT 5.6 Sol and GPT 6 Astra. The author provided the initial prompt and specified the mathematical technique to be used, after which the models developed the analysis and derived the reported results. The main result was derived by GPT 5.6 Sol and later shortened by GPT 6 Astra. The author then reviewed, edited, and reorganized the generated material to improve its clarity and readability, again with assistance from these models. The author takes responsibility for the paper’s final content.

The author also thanks Herbert Koch and Peer Kunstmann for helpful discussions concerning the results and presentation of this paper. The author was funded by the Deutsche Forschungsgemeinschaft
(DFG, German Research Foundation) – Project-ID 258734477 – SFB 1173 when he developed the ideas used in this paper.

\section{Modulation spaces and local well-posedness}\label{sec:local}

Our Fourier convention is
\[
 \widehat f(\xi)=\int_\R e^{-ix\xi}f(x)\dd x,
 \qquad f(x)=\frac1{2\pi}\int_\R e^{ix\xi}\widehat f(\xi)\dd\xi,
\]
initially for Schwartz functions and then for tempered distributions.
We write $D=-i\partial_x$ and $a(D)$ for the Fourier multiplier with
symbol $a$.  The notation $X\lesssim_\alpha Y$ means $X\le C_\alpha Y$,
and $X\asymp_\alpha Y$ means that both inequalities hold.  Unsubscripted
norms are over $\R$, and $[s]_+=\max(s,0)$.  We also use
\begin{equation}\label{eq:basic-notation}
 \langle x\rangle=(1+|x|^2)^{1/2},\qquad
 D_t^+f(t)=\limsup_{h\downarrow0}\frac{f(t+h)-f(t)}h.
\end{equation}
Thus $D_t^+$ denotes the upper right Dini derivative.
The partition in \eqref{eq:Mdef} can be constructed by choosing a
nonnegative even $\psi\in C_c^\infty((-2,2))$, positive on $[-1,1]$,
and setting
\[
 \varphi(\xi)=\frac{\psi(\xi)}{\sum_{j\in\Z}\psi(\xi-j)}.
\]
The denominator is positive, smooth, and $1$-periodic, so this gives
all the stated properties of $\varphi$.

We first establish local well-posedness in
$M=M^{\infty,1}(\R)$.  Its blow-up alternative reduces
Theorem~\ref{thm:main} to an a priori $M$-bound on every compact time
interval.  The contraction argument needs that $M$ is a Banach algebra,
embeds into $L^\infty$, and is stable under the free flow and the NLS
dilation, so we record these estimates first.  They are standard; see, for
example, \cite{BGOR2007,Klaus2023}.

The embedding and algebra estimate below ensure that the nonlinearity
in \eqref{eq:duhamel} is an ordinary bounded function and belongs to $M$.
The Duhamel formula follows from
$(\partial_t-i\partial_x^2)q=-2i|q|^2q$ by variation of constants.

\begin{lemma}[Elementary modulation estimates]\label{lem:modulation}
The space $M$ is complete.  Moreover, there is a constant $C$ such that
\begin{align}
 \norm{fg}_M&\le C\norm f_M\norm g_M,\label{eq:Malgebra}\\
 \norm f_{L^\infty}&\le C\norm f_M,\label{eq:Membed}\\
 \norm{e^{it\partial_x^2}f}_M
 &\le C\langle t\rangle^{1/2}\norm f_M.\label{eq:SchrM}
\end{align}
The Schr\"odinger group is strongly continuous on $M$.  Moreover, for
$0<\lambda\le1$,
\begin{equation}\label{eq:dilationM}
 \norm{S_\lambda f}_M\le C\lambda\norm f_M,
 \qquad S_\lambda f(x)=\lambda f(\lambda x).
\end{equation}
For fixed $\lambda>0$, $S_\lambda$ and $S_\lambda^{-1}$ are bounded on $M$.
\end{lemma}

\begin{proof}
We only prove \eqref{eq:dilationM} and refer to \cite{BGOR2007,Klaus2023} for a proof of the other statements, which are standard results. 
For $0<\lambda\le1$, the function
$\lambda(\square_jf)(\lambda\cdot)$ has Fourier support in
$[\lambda j-2\lambda,\lambda j+2\lambda]$ and therefore meets at most a
fixed number of output unit boxes, independently of $j$ and $\lambda$.
Its $L^\infty$ norm is at most
$\lambda\norm{\square_jf}_\infty$.  Although many input boxes may fall into
one output box, summing first over output boxes and then over input boxes
counts each input box only a fixed number of times.  This proves
\eqref{eq:dilationM}.  For fixed $\lambda>1$, an input interval meets
$O(1+\lambda)$ output boxes, which gives a finite constant depending on
$\lambda$.  Applying this statement also to $1/\lambda$ proves boundedness
of both $S_\lambda$ and its inverse.
\end{proof}

Lemma~\ref{lem:modulation} makes the Duhamel map contractive on a ball whose
lifespan depends only on the size of the initial datum.  The resulting
continuation criterion is the bridge from the later determinant estimate to
global existence.

\begin{proposition}[Local theory]\label{prop:local}
For every $A>0$ there is $\delta=\delta(A)>0$ such that, if
$\norm{q_0}_M\le A$, equation \eqref{eq:nls} has a unique mild solution on
$[-\delta,\delta]$.  On this data ball the solution map is locally Lipschitz
into $C([-\delta,\delta];M)$.
If $(T_-,T_+)$ is the maximal lifespan, then
\begin{equation}\label{eq:blowupalternative}
 T_+<\infty\quad\Longrightarrow\quad
 \limsup_{t\uparrow T_+}\norm{q(t)}_M=\infty,
\end{equation}
and similarly at $T_-$.
\end{proposition}

\begin{proof}
Write $C_tM=C([-\delta,\delta];M)$.  On this space use the right side of
\eqref{eq:duhamel} as the fixed-point map $\Phi_{q_0}$.  Lemma~\ref{lem:modulation} gives,
on the ball of radius
$2C A$,
\[
 \norm{\Phi_{q_0}(q)}_{C_tM}\le CA+C\delta\norm q_{C_tM}^3
\]
and, for two data and two functions in the same ball,
\[
 \begin{aligned}
 \norm{\Phi_{q_0}(q)-\Phi_{\widetilde q_0}(\widetilde q)}_{C_tM}
 &\le C\norm{q_0-\widetilde q_0}_M\\
 &\quad+C\delta(\norm q_{C_tM}^2+\norm{\widetilde q}_{C_tM}^2)
       \norm{q-\widetilde q}_{C_tM}.
 \end{aligned}
\]
Here the same estimates hold for positive and negative times, because
\eqref{eq:SchrM} depends only on $|t|$ and a backward Duhamel integral is
estimated by its absolute length.  We decrease $\delta$ so that
$\delta\le1$.  The contraction theorem proves existence and uniqueness in
the chosen ball, as well as local Lipschitz dependence.  Any two mild
solutions are bounded on a sufficiently short common subinterval; applying
the same difference estimate there and iterating gives uniqueness in the
full class $C_tM$.  The same argument starting at any time gives a
lifespan depending only on the norm at that time.  This proves
\eqref{eq:blowupalternative}: if
$\norm{q(t_r)}_M\le A$ along $t_r\uparrow T_+<\infty$, local theory at each
$t_r$ gives a solution on $[t_r-\delta(A),t_r+\delta(A)]$.  Choose $r$ with
$t_r+\delta(A)>T_+$.  Uniqueness identifies this new solution with the old
one on their overlap, thereby extending the maximal solution past $T_+$, a
contradiction.
\end{proof}

Integer Galilean transformations align the equation with the unit frequency
boxes, while dilation makes the data small and rescales the time interval and
the observation window.  The two NLS symmetries used later are
\begin{align}
 (\cG_nq)(t,x)
 &=e^{-i(nx+n^2t)}q(t,x+2nt),\qquad n\in\Z,
 \label{eq:Galilean}\\
 p(s,x)&=\lambda q(\lambda^2s,\lambda x),\qquad \lambda>0.
 \label{eq:NLSscaling}
\end{align}
Both functions solve \eqref{eq:nls} whenever $q$ does; this follows by direct
substitution for smooth solutions and then from the Duhamel formula for mild
solutions.  In particular, the second transformation has initial value
$p(0)=S_\lambda q_0$.

The conservation law will first be proved for smooth solutions and then
passed to endpoint data.  We therefore record the required approximation and
regularity-persistence facts here, before they are used.

\begin{lemma}[Finite-frequency approximation]
\label{lem:frequency-truncation}
For $q\in M$ define the frequency cutoff
\begin{equation}\label{eq:finite-frequency-truncation}
 q^{(M_0)}=\sum_{|j|\le M_0}\square_jq.
\end{equation}
Then $q^{(M_0)}\to q$ in $M$ as $M_0 \to \infty$, and the cutoff maps are uniformly bounded on
$M$: for a constant $C_\varphi$ depending only on the fixed partition,
\begin{equation}\label{eq:truncation-uniform-bound}
 \sup_{M_0}\norm{q^{(M_0)}}_M\le C_\varphi\norm q_M.
\end{equation}
Every $q^{(M_0)}$ is smooth and bounded together with all its derivatives.
The NLS solutions with these initial data converge to the solution with datum
$q$ in $C(J;M)$ on their common contraction interval $J$, and on any longer
compact interval on which they have a common $M$ bound.
\end{lemma}

\begin{proof}
The first assertion is the vanishing of the $\ell^1$ tail in
\eqref{eq:Mdef}; the bounded overlap of the boxes changes only a fixed
constant and also gives \eqref{eq:truncation-uniform-bound}.  Compact Fourier
support and the convolution form of Bernstein's inequality give, for every
integer $r\ge0$,
\[
 \norm{\partial_x^rq^{(M_0)}}_\infty
 \le C_{r,M_0}\sum_{|j|\le M_0}\norm{\square_jq}_\infty<\infty.
\]
The difference estimate in Proposition~\ref{prop:local}
gives convergence on a time interval depending only on a common bound for
the data.  Iterating the same estimate proves the longer-interval assertion.
\end{proof}

For completeness, note that smoothness persists for as long as the $M$-solution
exists:  If, for $s\ge0$,
$\norm f_{M_s}=\sum_j\langle j\rangle^s\norm{\square_jf}_\infty$, then
\[
 \langle j_1+j_2+j_3\rangle^s
 \le C_s\sum_{r=1}^3\langle j_r\rangle^s
\]
and the same frequency convolution as in \eqref{eq:Malgebra} give
\[
 \norm{fgh}_{M_s}\le C_s\bigl(
 \norm f_{M_s}\norm g_M\norm h_M
 +\norm f_M\norm g_{M_s}\norm h_M
 +\norm f_M\norm g_M\norm h_{M_s}\bigr).
\]
The proof of \eqref{eq:SchrM} does not move frequency supports, so the same
argument after multiplication by $\langle j\rangle^s$ gives
$\norm{e^{it\partial_x^2}f}_{M_s}\lesssim_s
\langle t\rangle^{1/2}\norm f_{M_s}$.
On every interval on which $\norm q_M$ is bounded, the Duhamel estimate in
$M_s$ and subdivision into sufficiently short intervals give a Gronwall
bound for $\norm q_{M_s}$.  Thus no $M_s$ norm can blow up before the
$M$-lifespan ends.  All classical identities used below therefore apply to
the finite-frequency approximants throughout their lifespan.

Finite frequency support gives smooth bounded functions, but not spatial
decay.  The differential conservation identities below are therefore first
checked locally (equivalently, after spatial mollification).  Whenever an
integration by parts in $x$ is used, one may insert a compactly supported
cutoff, do the calculation, and pass to the stated windowed or distributional
identity.  The uniformly local estimates of Section~\ref{sec:windows}
control the terms in this passage; no decay at spatial infinity is assumed.

\section{Windows and uniformly local spaces}\label{sec:windows}

Functions in $M^{\infty,1}$ need not decay, so the global $L^2$ quantities
used in the determinant argument may be infinite.  We therefore localize
with unnormalized spatial windows and develop a functional calculus on
uniformly local $L^2$; these tools make the later graded features meaningful
for endpoint data.

Fix throughout a positive, even function $\chi\in\mathcal S(\R)$ which is
decreasing on $[0,\infty)$, bounded below on $[-2,2]$, and satisfies
\begin{equation}\label{eq:windowderivatives}
 |\chi'(x)|+|\chi''(x)|\le C_\chi\chi(x).
\end{equation}
Such a window exists: for example,
$\chi(x)=e^{-\sqrt{1+x^2}}$ has all these properties.  We also fix a
positive Schwartz window $\widetilde\chi$ which is bounded
below on a sufficiently large fixed neighborhood of the origin.  It is
used when a unit-scale multiplier slightly enlarges a localization; its
precise choice is immaterial.  We write
$\widetilde\chi_y(x)=\widetilde\chi(x-y)$.
Only positivity on a fixed core, rapid decay, and
\eqref{eq:windowderivatives} are used for $\chi$; monotonicity is a convenient
way to organize the coverings.  The window $\widetilde\chi$ serves only for
fixed unit-scale enlargements and may be replaced by a sufficiently wide
translate-independent multiple of $\chi$.  Neither window is the frequency
cutoff $\psi$ used in Section~\ref{sec:main}.
For $R\ge1$ and $y\in\R$ set
\begin{equation}\label{eq:window}
 \chi_{R,y}(x)=\chi\left(\frac{x-y}{R}\right).
\end{equation}
The window is deliberately not normalized by a power of $R$.
For a function $F$ with values in a Hilbert space $H$, put
\begin{align*}
 \norm F_{U_R}&=\sup_y\norm{\chi_{R,y}F}_{L^2},\\
 \norm F_{I_R}&=\sup_y\norm{\one_{[y-R,y+R]}F}_{L^2}.
\end{align*}
We write $L^2_{\uloc}(\R;H)$ for the space on which
$\norm F_{L^2_{\uloc}}:=\norm F_{U_1}$ is finite.

The first lemma of this section shows that smooth and sharp uniformly local norms are
equivalent and, frequency by frequency, identifies local $L^2$ mass with the
modulation norm.  This is how the determinant features will eventually
recover $\norm q_M$.

\begin{lemma}[Window comparison]\label{lem:windows}
Uniformly for $R\ge1$,
\begin{equation}\label{eq:windowcomparison}
 c\norm F_{I_R}\le\norm F_{U_R}\le C\norm F_{I_{2R}},
 \qquad
 \norm F_{U_R}^2\le CR\norm F_{U_1}^2.
\end{equation}
Moreover, for a scalar function $f$,
\begin{equation}\label{eq:modulation-uloc}
 \sum_j\sup_y\norm{\chi_{1,y}\square_jf}_{L^2}
 \asymp_\chi\norm f_M.
\end{equation}
\end{lemma}

\begin{proof}
Let $c_0=\min_{|s|\le1}\chi(s)>0$.  Then
$\chi_{R,y}\ge c_0$ on $[y-R,y+R]$, which proves the lower bound in
\eqref{eq:windowcomparison}.  For the reverse bound, use the intervals
$J_k=[y+kR,y+(k+1)R]$.  Each $J_k$ is contained in an interval of radius
$2R$, so its $L^2$ mass is at most $\norm F_{I_{2R}}^2$.  For every $N>2$,
\[
 \sup_{x\in J_k}\chi_{R,y}(x)^2\le C_N\langle k\rangle^{-N}.
\]
Summation in $k$ gives the desired estimate, with constants independent of
$R$ and $y$.  A decomposition into unit intervals gives
\[
 \int\chi_{R,y}^2|F|^2
 \le C\norm F_{U_1}^2
       \sum_{r\in\Z}\left\langle\frac rR\right\rangle^{-N}
 \le CR\norm F_{U_1}^2.
\]
This proves the window bounds.  To prove \eqref{eq:modulation-uloc}, put
$f_j=\square_jf$ and choose $\psi\in C_c^\infty$ equal to one on
$\supp\varphi$.  If
\[
 K_j=\mathcal F^{-1}\bigl(\psi(\cdot-j)\bigr),
\]
then $f_j=K_j*f_j$ and $K_j(x)=e^{ijx}K_0(x)$.  For
$I_k(x)=[x+k,x+k+1]$, the uniform Schwartz bounds for $K_j$ give, for
$N>2$,
\begin{align*}
 |f_j(x)|
 &\le \sum_{k\in\Z}
   \norm{K_j(x-\cdot)}_{L^2(I_k(x))}
   \norm{f_j}_{L^2(I_k(x))}\\
 &\le C_N\sum_{k\in\Z}\langle k\rangle^{-N}
   \sup_y\norm{\chi_{1,y}f_j}_2
 \le C\sup_y\norm{\chi_{1,y}f_j}_2.
\end{align*}
Here the lower bound for $\chi$ on a fixed neighborhood of the origin was
used in the second line.  Conversely,
\[
 \sup_y\norm{\chi_{1,y}f_j}_2
 \le \norm\chi_2\norm{f_j}_\infty.
\]
Summing these two inequalities in $j$ proves
\eqref{eq:modulation-uloc}.
\end{proof}

Later we apply rational functions of first-order graded operators to sources
which are only uniformly locally square integrable.  Finite propagation lets
us define these multipliers by spatial exhaustion without imposing decay.

\begin{lemma}[Uniformly local functional calculus]\label{lem:ulocfc}
Let $H$ be a separable complex Hilbert space and let $\mathcal B(H)$ denote
its bounded operators.  Let $V\in\mathcal B(H)$ be a fixed self-adjoint
operator satisfying
$0\le V\le I_H$ in the quadratic-form sense, where $I_H$ is the identity on
$H$.  Let $W:\R\to\mathcal B(H)$ be strongly measurable and essentially
bounded, with $W(x)=W(x)^*$ almost everywhere, and let
\[
 A=-iV\partial_x+W(x).
\]
Then the group $e^{-itA}$ has
propagation speed at most one, in the sense that for all Borel sets
$E,K\subset\R$,
\begin{equation}\label{eq:finite-propagation}
 \one_Ee^{-itA}\one_K=0
 \qquad\text{whenever}\qquad \operatorname{dist}(E,K)>|t|,
\end{equation}
and satisfies the bound
\begin{equation}\label{eq:ulocpropagation}
 \norm{e^{-itA}F}_{I_R}
 \le C\left(1+\frac{|t|}{R}\right)^{1/2}\norm F_{I_R}.
\end{equation}
If
\[
 a(s)=\int_\R e^{-its}\dd\mu_a(t)
\]
for a finite measure $\mu_a$ satisfying
\[
 \int_\R(1+|t|)^{1/2}\dd|\mu_a|(t)<\infty,
\]
then the global $L^2$ operator $a(A)$ has a canonical bounded realization
on $L^2_{\uloc}$.  It is defined by cutoff exhaustion, is independent of
that exhaustion, and satisfies
\begin{equation}\label{eq:uloccalculus}
 \norm{a(A)F}_{I_R}
 \le C\int_\R\left(1+\frac{|t|}{R}\right)^{1/2}
       \dd|\mu_a|(t)\norm F_{I_R}.
\end{equation}
The same estimates hold with $I_R$ replaced by $U_R$, with input and output
radii changed by at most a fixed factor independent of $R$.
\end{lemma}

\begin{proof}
We use the convention $\langle u,v\rangle=\int\overline u\,v$, so the
Hilbert-space inner product is linear in its second argument.  For $A$ the free part is the self-adjoint Fourier multiplier $\xi V$ with
domain $\{F:\xi V\widehat F\in L^2\}$, and $A$ is self-adjoint on this
domain by the bounded perturbation theorem; see
\cite[Chapter~VIII]{ReedSimon1980}. The multiplier
$\xi V$ is self-adjoint because its symbol is
self-adjoint for every real $\xi$; the possible kernel of $V$ merely produces
a zero multiplier on that subspace and causes no domain problem.  The bounded
perturbation theorem therefore applies even when $V$ is not invertible.

Let $u(t)=e^{-itA}F$.  For every bounded real Lipschitz function $\phi$,
\begin{equation}\label{eq:local-energy}
 \frac{\dd}{\dd t}\int\phi(x)|u(t,x)|_H^2\dd x
 =\int\phi'(x)\ip{Vu(t,x)}{u(t,x)}_H\dd x.
\end{equation}
For domain data and a smooth bounded $\phi$, this follows from
$u_t=-V u_x-iWu$ by integration by parts; the $W$ terms cancel because
$W=W^*$.  More explicitly,
$2\operatorname{Re}\langle u,-iWu\rangle_H=0$, while
$-2\operatorname{Re}\langle u,Vu_x\rangle_H
=-\partial_x\langle Vu,u\rangle_H$.
Here the integration by parts is justified in the graph domain as follows.
For $v\in\operatorname{Dom}(A_0)$, first convolve in $x$ with a smooth
compactly supported approximate identity $\rho_\epsilon$.  Both
$v*\rho_\epsilon\to v$ and
$V\partial_x(v*\rho_\epsilon)=(V\partial_xv)*\rho_\epsilon
\to V\partial_xv$ in $L^2$.  For fixed $\epsilon$, multiply by
$\eta(x/L)$, where $\eta\in C_c^\infty$ equals one near the origin.
As $L\to\infty$, this converges in the same graph norm: the additional
derivative term is bounded in $L^2$ by
$L^{-1}\|\eta'\|_\infty\|V\|\|v*\rho_\epsilon\|_2$.
Choosing $\epsilon\downarrow0$ and then sufficiently large $L$ gives
smooth compactly supported approximants.  Boundedness of $W$ makes the
graph norms of $A$ and $A_0$ equivalent, so the approximation also works
for $A$.  Multiplication by a bounded Lipschitz $\phi$ preserves this
domain, with
$V\partial_x(\phi v)=\phi V\partial_xv+\phi'Vv$.
Thus the integration-by-parts identity passes to every domain vector;
applied to $u(t)$ it proves \eqref{eq:local-energy} for domain data.

Approximate a bounded Lipschitz $\phi$ by smooth mollifications, with
uniform bounds on both the functions and their derivatives.  The
derivatives converge almost everywhere, so dominated convergence gives
the same identity in its time-integrated form.  Finally approximate
arbitrary $F\in L^2$ by domain data.  Unitarity gives convergence uniformly
in time in $L^2$, and boundedness of $\phi$ and $\phi'V$ permits passage
to the integrated identity.  Its right-hand integrand is continuous in
time by strong continuity of the group, which recovers the displayed
derivative identity.  Only boundedness and self-adjointness of $W$ are
used; no regularity of the measurable coefficient is required.
Cutoffs converging to
the indicator of a backward light cone prove finite propagation as follows.
First suppose that $F$ is supported in a closed set $K$, put
$d_K(x)=\operatorname{dist}(x,K)$, and let $\eta$ be a smooth nondecreasing
approximation to $\one_{(0,\infty)}$.  For $s\ge0$ use the time-dependent
weight $\phi_s(x)=\eta(d_K(x)-s)$.  Combining
\eqref{eq:local-energy} with $\partial_s\phi_s=-\eta'(d_K-s)$ gives
\[
 \frac{\dd}{\dd s}\int\phi_s|u(s)|_H^2\dd x
 =\int\eta'(d_K-s)
   \left[d_K'(x)\ip{Vu}{u}_H-|u|_H^2\right]\dd x\le0,
\]
because $|d_K'|\le1$ and $0\le V\le I_H$.  The expression vanishes at
$s=0$; approximation of the distance function and of the step function
therefore shows that $u(s)$ vanishes where $d_K>s$.  Taking adjoints gives
the same conclusion for negative times, and density proves
\eqref{eq:finite-propagation}.

Now let $J=[y-R,y+R]$ and
$J^{|t|}=[y-R-|t|,y+R+|t|]$.  By
\eqref{eq:finite-propagation} and unitarity,
\begin{align*}
 \norm{\one_Je^{-itA}F}_2^2
 &=\norm{\one_Je^{-itA}(\one_{J^{|t|}}F)}_2^2
 \le\norm{\one_{J^{|t|}}F}_2^2\\
 &\le C\left(1+\frac{|t|}{R}\right)\norm F_{I_R}^2,
\end{align*}
where the last line follows by covering $J^{|t|}$ by
$O(1+|t|/R)$ intervals of length $2R$.  Taking the supremum over $y$
proves \eqref{eq:ulocpropagation}.

For compactly supported $F$ define
\[
 a(A)F=\int e^{-itA}F\dd\mu_a(t).
\]
Minkowski's inequality and \eqref{eq:ulocpropagation} give
\eqref{eq:uloccalculus}.  For general uniformly local $F$, truncate first
in $x$ and then in the $t$-integral.  Finite propagation makes the truncated
operators Cauchy on compact spatial sets, while the integrable measure tail
controls the remainder.  More precisely, let $\eta_L\in C_c^\infty$ satisfy
$0\le\eta_L\le1$, $\eta_L=1$ on $[-L,L]$, and
$\supp\eta_L\subset[-2L,2L]$.  Given a compact set $K$ and $T>0$, if
$L,L'$ are sufficiently
large, then \eqref{eq:finite-propagation} gives
\[
 \one_K\int_{|t|\le T}e^{-itA}
       (\eta_L-\eta_{L'})F\dd\mu_a(t)=0.
\]
On the other hand, uniformly in $L$,
\[
 \norm{\one_K\int_{|t|>T}e^{-itA}\eta_LF\dd\mu_a(t)}_2
 \le C_K\norm F_{I_1}
       \int_{|t|>T}(1+|t|)^{1/2}\dd|\mu_a|(t),
\]
and the right side tends to zero as $T\to\infty$.  This gives an
exhaustion-independent limit.  It is the unique realization compatible with
all such cutoff exhaustions; no density of compactly supported functions in
$L^2_{\uloc}$ is asserted.  Lemma~\ref{lem:windows} transfers the estimates
to $U_R$ and also shows that $I_1$ gives an equivalent definition of
$L^2_{\uloc}$.  In particular, the limit is an equivalence class in
$L^2_{\rm loc}$; addition and scalar multiplication commute with every
cutoff limit, so the realization is linear.  The constants in the $U_R$
version depend only on the fixed window and on
$\int(1+|t|)^{1/2}\dd|\mu_a|$, not on $R$ or $y$.  Since the sharp norm
$I_1$ is equivalent to $U_1$, this also shows directly that
$L^2_{\uloc}$ is a Banach space.
\end{proof}

The determinant multipliers are products of three resolvents.  We next show
that the factorization, inverse identities, and cutoff limits survive on
nondecaying uniformly local data. Note that the superscript $(2)$ in the Lemma records that
$R_{a,\epsilon}^{(2)}$ is initially the usual resolvent on global $L^2$;
it is dropped for its uniformly local extension.

\begin{lemma}[Factorized resolvents on uniformly local data]
\label{lem:uloc-resolvent-factorization}
Let $A=-iV\partial_x+W(x)$ be as in Lemma~\ref{lem:ulocfc}, let
$c\in\R$, and put $B=A+c$.  For $a>0$ and
$\epsilon\in\{-1,1\}$, the global $L^2$ resolvent
\[
 R_{a,\epsilon}^{(2)}=(a+i\epsilon B)^{-1}
\]
has a canonical cutoff-compatible bounded realization $R_{a,\epsilon}$ on
$L^2_{\uloc}(\R;H)$.   It has the following
properties.
\begin{enumerate}[label=\textup{(\roman*)}]
\item For every $f\in L^2_{\uloc}$,
\begin{equation}\label{eq:uloc-resolvent-inverse}
 (a+i\epsilon B)R_{a,\epsilon}f=f
 \quad\hbox{in }\mathcal D'(\R;H).
\end{equation}
Moreover, if $v\in L^2_{\uloc}$ and
$(a+i\epsilon B)v=0$ in $\mathcal D'(\R;H)$, then $v=0$.
\item The resolvents $R_{a,\epsilon}$ with the same $\epsilon$ commute.
If
\[
 h_{\kappa,\epsilon}(s)
 =\frac{8\kappa^2}
 {(2\kappa+i\epsilon s)(4\kappa+i\epsilon s)
  (6\kappa+i\epsilon s)},
\]
then
\begin{equation}\label{eq:uloc-h-factorization}
 h_{\kappa,\epsilon}(B)
 =8\kappa^2R_{2\kappa,\epsilon}R_{4\kappa,\epsilon}
 R_{6\kappa,\epsilon}
\end{equation}
on $L^2_{\uloc}$, and this multiplier is injective there.
\item If $\eta_L\in C_c^\infty$, $0\le\eta_L\le1$, equals one on
$[-L,L]$, and is supported in $[-2L,2L]$, then for every compact
$K\subset\R$,
\begin{align}
 R_{a,\epsilon}^{(2)}(\eta_Lf)&\longrightarrow R_{a,\epsilon}f,
 \label{eq:resolvent-exhaustion}\\
 h_{\kappa,\epsilon}(B)(\eta_Lf)&\longrightarrow
 h_{\kappa,\epsilon}(B)f
 \label{eq:h-exhaustion}
\end{align}
in $L^2(K;H)$.  The same statement holds for every finite product
of the resolvents and the limits are independent of the exhaustion.
\end{enumerate}
\end{lemma}

\begin{proof}
Since $B$ is self-adjoint and $|a+i\epsilon s|\ge a$ for real $s$, the
global resolvent exists.  The scalar identity
\[
 (a+i\epsilon s)^{-1}=\int_0^\infty e^{-at}e^{-i\epsilon ts}\dd t
\]
and the spectral theorem give the Laplace representation
\begin{equation}\label{eq:resolvent-laplace}
 R_{a,\epsilon}^{(2)}
 =\int_0^\infty e^{-at}e^{-i\epsilon tB}\dd t
\end{equation}
as a strong operator integral on $L^2$.  Adding the scalar $c$, changing
the sign of time, and changing the bounded potential do not change
the propagation speed.  Lemma~\ref{lem:ulocfc} therefore gives, for every
compact interval $K$,
\begin{equation}\label{eq:local-group-bound-for-resolvent}
 \norm{\one_Ke^{-i\epsilon tB}F}_2
 \le C_K(1+|t|)^{1/2}\norm F_{U_1}.
\end{equation}

For $f\in L^2_{\uloc}$, define $R_{a,\epsilon}f$ as the local limit of
$R_{a,\epsilon}^{(2)}(\eta_Lf)$.  To verify convergence, suppose $L>L'$
and put
\[
 d_{L',K}=\operatorname{dist}
 \bigl(K,\supp(\eta_L-\eta_{L'})\bigr).
\]
Finite propagation, \eqref{eq:resolvent-laplace}, and
\eqref{eq:local-group-bound-for-resolvent} imply
\begin{align}
 &\norm{\one_KR_{a,\epsilon}^{(2)}
       ((\eta_L-\eta_{L'})f)}_2\notag\\
 &\qquad\le C_K\norm f_{U_1}
 \int_{d_{L',K}}^\infty e^{-at}(1+t)^{1/2}\dd t.
 \label{eq:resolvent-cutoff-tail}
\end{align}
The right side tends to zero uniformly for $L>L'$.  This proves the
existence and cutoff-independence of the limit, as well as
$\norm{R_{a,\epsilon}f}_{U_1}\le C_a\norm f_{U_1}$.

On global $L^2$,
$(a+i\epsilon B)R_{a,\epsilon}^{(2)}F=F$.  Pass to the cutoff limit.
The operator $B$, being first order with bounded zeroth-order
coefficients, is continuous from $L^2_{\rm loc}$ into vector-valued
distributions.  This proves \eqref{eq:uloc-resolvent-inverse}.  We also need
the converse uniqueness statement.  Suppose
$(a+i\epsilon B)v=0$ in distributions and $v\in L^2_{\uloc}$.  Fix
$x_0\in\R$ and choose
\[
 \omega(x)=e^{-\delta\langle x-x_0\rangle},\qquad 0<\delta<a.
\]
Then $\omega v\in L^2$.  Take the real part of the equation paired with
$\omega^2v$.  The self-adjoint zeroth-order terms contribute only
imaginary quantities, whereas integration by parts in the transport term
gives
\[
 a\int\omega^2|v|_H^2
 =\epsilon\int\omega\omega'\ip{Vv}{v}_H.
\]
Because $0\le V\le I$ and $|\omega'|\le\delta\omega$, the absolute value
of the right side is at most
$\delta\int\omega^2|v|_H^2$.  Hence $v=0$.  To justify the calculation for a
distributional solution, convolve the equation in $x$, test the regularized
equation against a compactly truncated $\omega^2v$, and pass first in the
mollification and then in the cutoff.  The constant coefficient $V$ commutes
with convolution, while both $W(v*\rho_\nu)$ and $(Wv)*\rho_\nu$ converge to
$Wv$ in $L^2_{\rm loc}$.  Finally, $v\in L^2_{\uloc}$ implies
$\omega v\in L^2$, so the limiting boundary terms vanish.
For a compact interval $K$, the two convergence statements are simply
\[
 \|W(v*\rho_\nu-v)\|_{L^2(K)}
 \le\norm W_\infty\|v*\rho_\nu-v\|_{L^2(K)}\to0,
\]
after enlarging $K$ slightly, and
$\|(Wv)*\rho_\nu-Wv\|_{L^2(K)}\to0$ by the usual local
$L^2$ approximation identity.  Their difference is exactly the measurable
coefficient commutator that occurs in the mollified equation.

For $a,b>0$, Fubini's theorem in \eqref{eq:resolvent-laplace} gives on
global $L^2$
\[
 R_{a,\epsilon}^{(2)}R_{b,\epsilon}^{(2)}F
 =\int_0^\infty\!\int_0^\infty
 e^{-as-bt}e^{-i\epsilon(s+t)B}F\dd s\dd t.
\]
This is symmetric under interchange of $(a,s)$ and $(b,t)$.  The same
cutoff-tail argument, now with the convolution of the two exponential
measures, proves commutation and cutoff compatibility on $L^2_{\uloc}$.
Indeed, after restriction to a compact interval the integrand is bounded by
$C_K(1+s+t)^{1/2}e^{-as-bt}\norm F_{U_1}$, which is integrable on
$[0,\infty)^2$; this is the precise domination needed for Fubini.
It applies without change to every finite product.

Finally, the scalar factorization of $h_{\kappa,\epsilon}$, the spectral
theorem, and cutoff exhaustion prove \eqref{eq:uloc-h-factorization}.
Successive use of \eqref{eq:uloc-resolvent-inverse} gives
\begin{align*}
 &(6\kappa+i\epsilon B)(4\kappa+i\epsilon B)
 (2\kappa+i\epsilon B)h_{\kappa,\epsilon}(B)f
 =8\kappa^2f
\end{align*}
in distributions.  Therefore $h_{\kappa,\epsilon}(B)$ is injective.  Each
scalar resolvent has the one-sided exponential measure in the Laplace formula
above.  Hence the Fourier measure of $h_{\kappa,\epsilon}$ is the convolution
of three such measures and has every polynomial moment.  The same tail
argument proves \eqref{eq:h-exhaustion}.
The differential factors in the preceding display are legitimate: from
\eqref{eq:uloc-resolvent-inverse}, a resolvent output has the distributional
$B$-regularity needed to apply its corresponding factor, and the factors are
removed one at a time in the reverse order.  We never claim that an arbitrary
$L^2_{\uloc}$ function lies in the global operator domain of the unbounded
product.
\end{proof}

\section{Nondecaying diagonal Green functions}\label{sec:green}

The positive determinant density is built from diagonal values of the
Zakharov--Shabat resolvent.  Here we construct those values for bounded,
possibly nondecaying potentials; because ordinary Jost asymptotics are then
unavailable, the construction uses half-line $L^2$ Weyl lines.

For $\lambda>0$ introduce the Zakharov--Shabat operator
\begin{equation}\label{eq:Lax}
 L_\lambda(q)=
 \begin{pmatrix}\lambda-\partial_x&q\\-\bar q&\lambda+\partial_x\end{pmatrix}
 =\lambda I+\mathcal A(q).
\end{equation}

We first prove that this operator has an everywhere-defined inverse with
enough $H^1$ regularity to evaluate its output pointwise.  This is the basis
for constructing the Green kernel.

\begin{lemma}[Accretive resolvent]\label{lem:accretive}
For $q\in L^\infty$, $\mathcal A(q)$ is skew-adjoint on
$H^1(\R;\C^2)$.  The operator $L_\lambda(q):H^1\to L^2$ is bijective and
\begin{align}
 \norm{L_\lambda(q)^{-1}}_{L^2\to L^2}&\le\lambda^{-1},
 \label{eq:L2-resolvent}\\
 \norm{L_\lambda(q)^{-1}}_{L^2\to H^1}
 &\le C_{\lambda,\norm q_\infty}.\label{eq:H1-resolvent}
\end{align}
If $q_j\to q$ in $L^\infty$ and $\sup_j\norm{q_j}_\infty\le Q$, then
\begin{equation}\label{eq:resolvent-continuity}
 \norm{L_\lambda(q_j)^{-1}-L_\lambda(q)^{-1}}_{L^2\to H^1}
 \le C_{\lambda,Q}\norm{q_j-q}_\infty.
\end{equation}
\end{lemma}

\begin{proof}
Let
\[
 A_0=\begin{pmatrix}-\partial_x&0\\0&\partial_x\end{pmatrix},
 \qquad
 Q(x)=\begin{pmatrix}0&q(x)\\-\overline{q(x)}&0\end{pmatrix}.
\]
On $L^2(\R;\C^2)$, $A_0$ has domain $H^1(\R;\C^2)$ and
$A_0^*=-A_0$, by the elementary adjoint identity
$\partial_x^*=-\partial_x$.  Multiplication by $Q$ is bounded and
$Q(x)^*=-Q(x)$ almost everywhere.  Therefore the bounded-perturbation
adjoint formula gives
\[
 \mathcal D((A_0+Q)^*)=\mathcal D(A_0^*)=H^1,\qquad
 (A_0+Q)^*=-(A_0+Q).
\]
Thus $\mathcal A(q)=A_0+Q$ is skew-adjoint; see also the bounded
perturbation theorem \cite[Chapter~VIII]{ReedSimon1980}.  In particular,
$L_\lambda(q)^*=\lambda I-\mathcal A(q)$ has exactly the same domain
$H^1$, a fact used in the range argument below.  To see explicitly
why the resolvent estimate does not
depend on the size of $q$, write $F=(f,g)^T\in H^1(\R;\C^2)$.  Integration by
parts gives
\[
\begin{split}
 \Rea\ip{F}{L_\lambda(q)F}
 &=\lambda\norm{F}_{L^2(\R;\C^2)}^2\\
 &\quad+\Rea\int_\R
   \bigl(-\bar f f_x+\bar g g_x+\bar f qg-\bar g\,\bar qf\bigr)\dd x\\
 &=\lambda\norm{F}_{L^2(\R;\C^2)}^2.
\end{split}
\]
Indeed, the derivative terms have zero real part, while, with
$z=\bar f qg$, the two potential terms are $z-\bar z$ and are therefore
purely imaginary.  Thus Cauchy--Schwarz yields
\begin{equation}\label{eq:accretive-lower-bound}
 \norm F_2\le \lambda^{-1}\norm{L_\lambda(q)F}_2.
\end{equation}
In particular, $L_\lambda(q)$ is injective.  Its range is closed as well: if
$L_\lambda(q)F_n$ converges in $L^2$, then
\eqref{eq:accretive-lower-bound} makes $F_n$ Cauchy in $L^2$, and the
closedness of $L_\lambda(q)$ identifies its limit as an element of the domain.
Moreover,
$L_\lambda(q)^*=\lambda I-\mathcal A(q)$ satisfies the same estimate, so
$\ker L_\lambda(q)^*=\{0\}$.  Since
\[
 \overline{\operatorname{Ran}L_\lambda(q)}
   =\bigl(\ker L_\lambda(q)^*\bigr)^\perp=L^2,
\]
the range is also dense and hence equals $L^2$.  This proves bijectivity;
applying \eqref{eq:accretive-lower-bound} to
$F=L_\lambda(q)^{-1}G$ proves \eqref{eq:L2-resolvent}.  Notice that no
$L^\infty$ bound for $q$ enters this argument beyond ensuring that the
off-diagonal multiplication operator is a bounded skew-adjoint perturbation.

The component equations imply
\[
 \norm{\partial_xF}_2
 \le\norm{L_\lambda(q)F}_2
      +(\lambda+\norm q_\infty)\norm F_2,
\]
which proves \eqref{eq:H1-resolvent} by choosing $F$ as above.  The identical component estimate
for $L_\lambda(q)^*$ also gives
\[
 \norm{(L_\lambda(q)^*)^{-1}}_{L^2\to H^1}
 \le C_{\lambda,\norm q_\infty}.
\]
 The resolvent identity
\[
 L_\lambda(q_j)^{-1}-L_\lambda(q)^{-1}
 =-L_\lambda(q_j)^{-1}\bigl(L_\lambda(q_j)-L_\lambda(q)\bigr)
   L_\lambda(q)^{-1}
\]
together with the $L^2\to H^1$ bound for the first factor and the
$L^2\to L^2$ bound for the last proves
\eqref{eq:resolvent-continuity}.
\end{proof}

A homogeneous solution will mean a solution of
\begin{equation}\label{eq:homogeneous-lax}
 \partial_x\binom{f_1}{f_2}
 =\begin{pmatrix}\lambda&q\\\bar q&-\lambda\end{pmatrix}
  \binom{f_1}{f_2},
\end{equation}
which is equivalent to $L_\lambda(q)F=0$.  The one-dimensional Sobolev
inequality $|f(x)|\le C\norm f_{H^1}$ makes point evaluation bounded on
$H^1(\R)$.  The next proposition uses this fact to construct the Green
kernel and fixes all sign conventions.

Because a bounded potential need not decay, Jost solutions characterized by
prescribed free asymptotics at $\pm\infty$ are not available in general.  We
therefore replace them by the corresponding Weyl lines, selected only by
half-line square integrability.  For the Jost-solution construction and the
corresponding Green-kernel formula in the decaying AKNS setting, see
\cite[Appendix~B]{KKL2023}.

\begin{proposition}[Weyl lines and the Green kernel]\label{prop:green}
For every $q\in L^\infty$ and $x\in\R$, there is a one-dimensional subspace
$E_-(x)\subset\C^2$ consisting of the values at $x$ of homogeneous solutions
square integrable on $(-\infty,x]$, and a one-dimensional subspace
$E_+(x)\subset\C^2$ consisting of the values at $x$ of homogeneous solutions
square integrable on $[x,\infty)$.  Choose nonzero global homogeneous
solutions $\ell=(a,b)^T$ and $r=(c,d)^T$ so that
\[
 E_-(x)=\operatorname{span}\{\ell(x)\},\qquad
 E_+(x)=\operatorname{span}\{r(x)\},
\]
and put $\Delta=ad-bc$.  Then $\Delta\ne0$ is constant and
\begin{equation}\label{eq:green-kernel}
 G_q(x,y;\lambda)=\frac1\Delta
 \begin{cases}
  \binom{a(x)}{b(x)}\bigl(d(y),c(y)\bigr),&x<y,\\[1mm]
  \binom{c(x)}{d(x)}\bigl(b(y),a(y)\bigr),&x>y.
 \end{cases}
\end{equation}
The kernel has the jump
\begin{equation}\label{eq:green-jump}
 \operatorname{diag}(-1,1)\,[G_q(y+,y)-G_q(y-,y)]=I.
\end{equation}
If $q_j\to q$ in $L^\infty$ with a common bound, then the one-sided
diagonal values of $G_{q_j}$ converge to those of $G_q$ uniformly in $x$.
\end{proposition}

\begin{proof}
Put $T=L_\lambda(q)^{-1}$.  For fixed $x\in\R$ and $i\in\{1,2\}$, the
Sobolev point-evaluation bound and \eqref{eq:H1-resolvent} give
\[
 |(TF)_i(x)|\le C\norm{TF}_{H^1}
 \le C_{\lambda,\norm q_\infty}\norm F_2.
\]
Thus $F\mapsto(TF)_i(x)$ is a bounded functional on
$L^2(\R;\C^2)$.  The Riesz representation theorem supplies an $L_y^2$
row, and the two rows define a measurable matrix kernel
$G_q(x,y;\lambda)$ such that
\begin{equation}\label{eq:green-operator-representation}
 (TF)(x)=\int_\R G_q(x,y;\lambda)F(y)\dd y,
\end{equation}
where $TF$ is taken in its continuous $H^1$ representative.  The same
estimate, uniformly in $x$, gives the row bound
\[
 \sup_{x\in\R}
 \norm{G_q(x,\cdot;\lambda)}_{L_y^2(\C^{2\times2})}
 \le C_{\lambda,\norm q_\infty}.
\]

The identical $L^2\to H^1$ estimate for
$T^*=(L_\lambda(q)^*)^{-1}$ implies by duality that
$T:H^{-1}\to L^2$ is bounded.  Since
$\delta_y e_j\in H^{-1}(\R;\C^2)$ with norm independent of $y$, define
the $j$th column of $G_q(\cdot,y;\lambda)$ to be
$T(\delta_y e_j)$.  Testing against compactly supported $L^2$ functions
shows that these columns agree almost everywhere with the rows in
\eqref{eq:green-operator-representation}, and
\[
 \sup_{y\in\R}
 \norm{G_q(\cdot,y;\lambda)}_{L_x^2(\C^{2\times2})}
 \le C_{\lambda,\norm q_\infty}.
\]
In distributions,
$L_{\lambda,x}G_q(x,y;\lambda)=\delta_yI$.  Hence every column belongs to
$H^1_{\rm loc}$ away from $x=y$, solves \eqref{eq:homogeneous-lax} there,
and has one-sided limits at the diagonal.  To compute the jump, write
$J=\operatorname{diag}(-1,1)$ and
$V_\lambda(x)=\left(\begin{smallmatrix}\lambda&q(x)\\
-\bar q(x)&\lambda\end{smallmatrix}\right)$, so that
$L_\lambda=J\partial_x+V_\lambda$.  Integrating over
$(y-\varepsilon,y+\varepsilon)$ gives
\[
 J\bigl[G_q(y+\varepsilon,y;\lambda)
       -G_q(y-\varepsilon,y;\lambda)\bigr]
 +\int_{y-\varepsilon}^{y+\varepsilon}
    V_\lambda(x)G_q(x,y;\lambda)\dd x=I.
\]
The integral tends to zero because $V_\lambda$ is bounded and each column
of $G_q(\cdot,y;\lambda)$ is locally integrable.  Letting
$\varepsilon\downarrow0$ yields \eqref{eq:green-jump}, or explicitly
\[
 G_q(y+,y;\lambda)-G_q(y-,y;\lambda)
 =\begin{pmatrix}-1&0\\0&1\end{pmatrix}.
\]
Here
\[
 G_q(y-,y;\lambda):=\lim_{x\uparrow y}G_q(x,y;\lambda),\qquad
 G_q(y+,y;\lambda):=\lim_{x\downarrow y}G_q(x,y;\lambda);
\]
the sign records the side from which the first spatial variable approaches
$y$.  Write
\[
 L_y=G_q(y-,y;\lambda),\qquad R_y=G_q(y+,y;\lambda).
\]
The columns of $L_y$ and $R_y$ extend to half-line $L^2$ solutions: they are
restrictions of the globally $L^2$ columns
$T(\delta_y e_j)$ and solve the homogeneous equation away from $y$.  Any
two half-line $L^2$ solutions $u,v$ at the same end are dependent.  Indeed, the
coefficient matrix in \eqref{eq:homogeneous-lax} has trace zero, so
$(u_1v_2-u_2v_1)'=0$.  Since $|u|^2+|v|^2$ is integrable on the relevant
half-line, divide that half-line into disjoint unit intervals $I_n$ tending
to the end.  Their integrals tend to zero, and choose $x_n\in I_n$ so that
\[
 |u(x_n)|^2+|v(x_n)|^2
 \le \int_{I_n}(|u|^2+|v|^2).
\]
This is a common sequence on which both solutions tend to zero.
The constant Wronskian therefore vanishes.  Thus both trace matrices have
rank at most one.
Neither can have rank zero, because
\[
 R_y-L_y=\operatorname{diag}(-1,1)
\]
by \eqref{eq:green-jump}; if one trace vanished, the other would have rank
two.  Hence both ranks are exactly one, and their ranges construct
$E_-(y)$ and $E_+(y)$.  They also span \emph{all} half-line $L^2$
solutions: compare any such solution with a nonzero trace column and repeat
the Wronskian argument.  Their values at $y$ are dependent, and uniqueness
for the bounded-coefficient ODE makes the solutions dependent on the whole
half-line.  This proves both existence and one-dimensionality asserted in
the proposition.

The lines $E_-(y)$ and $E_+(y)$ are transverse.  Otherwise their common vector
would generate a global $L^2$ element of the kernel of $L_\lambda(q)$, contradicting
Lemma~\ref{lem:accretive}.  Since the coefficient matrix is traceless,
$\Delta'=0$; transversality gives $\Delta\ne0$.  On $x<y$, each column of
the kernel is a multiple of $\ell(x)$, while on $x>y$ it is a multiple of
$r(x)$.  Solving the resulting two-by-two linear system imposed by
\eqref{eq:green-jump} gives the row coefficients in
\eqref{eq:green-kernel}.  Here is that calculation explicitly.  Write
$G(x,y)=\ell(x)\alpha(y)$ for $x<y$ and
$G(x,y)=r(x)\beta(y)$ for $x>y$, with row vectors $\alpha,\beta$.
Since the jump equation is equivalent to
$r(y)\beta(y)-\ell(y)\alpha(y)=\operatorname{diag}(-1,1)$,
\[
 \begin{pmatrix}a&c\\b&d\end{pmatrix}
 \binom{-\alpha}{\beta}
 =\operatorname{diag}(-1,1),\qquad
 \begin{pmatrix}a&c\\b&d\end{pmatrix}^{-1}
 =\frac1\Delta\begin{pmatrix}d&-c\\-b&a\end{pmatrix}.
\]
Thus $\alpha=(d,c)/\Delta$ and $\beta=(b,a)/\Delta$, including precisely
the order and signs displayed in \eqref{eq:green-kernel}.

For stability, write
$G_j=G_{q_j}(\cdot,\cdot;\lambda)$, $G=G_q(\cdot,\cdot;\lambda)$,
$L_j=L_\lambda(q_j)$, and $L=L_\lambda(q)$.  The kernel form of the
resolvent identity is
\[
 G_j(x,y)-G(x,y)
 =-\int_\R G_j(x,z)(L_j-L)(z)G(z,y)\dd z.
\]
To justify its pointwise use, first compose the operator resolvent identity
with point evaluation at $x$.  This gives the displayed formula as an
$L_y^2$ identity; for fixed $x,y$ the integral on the right is absolutely
defined by the $L_z^2$ row and column bounds.  On each of the regions $x<y$
and $x>y$, both sides have the locally absolutely continuous
representatives supplied by the homogeneous ODE.  Equality almost everywhere
therefore extends to every point in each region, after which the one-sided
diagonal limits may be taken.
The row and column bounds just proved, followed by Cauchy--Schwarz, yield
\[
 \sup_{x,y}|G_j(x,y)-G(x,y)|
 \le C_{\lambda,Q}\norm{q_j-q}_\infty.
\]
Taking the one-sided limits proves the last assertion.  We record also
the quantitative nondegeneracy needed below.  Applying the same resolvent
identity with one potential equal to zero, and using the row and column
bounds above, gives
\[
 \sup_x\bigl(\norm{L_x}+\norm{R_x}\bigr)\le C_{\lambda,Q}
\]
whenever \(\norm q_\infty\le Q\).  The trace pairs therefore lie in the set
of pairs with both ranks at most one, fixed difference
\(\operatorname{diag}(-1,1)\), and norms at most \(C_{\lambda,Q}\).
This is a closed and bounded subset of a finite-dimensional matrix space,
hence compact.  Neither member of such a pair can vanish, because the other
would then equal the invertible jump.
The nonzero singular value of each trace therefore has a positive lower
bound depending only on $\lambda$ and $Q$.  It follows, for example by
normalizing a column of maximal length, that uniform convergence of the
traces implies uniform convergence of their ranges.
\end{proof}

\paragraph{The forward ratio.}
The kernel construction uses both Weyl lines, but the main proof requires
only the ratio on the left line.

Every homogeneous solution satisfies
\begin{equation}\label{eq:Jmonotonicity}
 \partial_x(|f_1|^2-|f_2|^2)=2\lambda(|f_1|^2+|f_2|^2).
\end{equation}
The potential terms cancel because their real parts agree.

\begin{proposition}[Forward normalized Green function]\label{prop:normalized}
The left Weyl component $a$ never vanishes.  The ratio
\begin{equation}\label{eq:gpm}
 g_{\lambda,+}=b/a
\end{equation}
is independent of the spanning solution and satisfies the Ricatti equation
\begin{equation}\label{eq:riccati-plus}
 (g_{\lambda,+})_x=\bar q-2\lambda g_{\lambda,+}-qg_{\lambda,+}^2.
\end{equation}
If $Q=\|q\|_\infty$, then
\begin{equation}\label{eq:Schur}
 |g_{\lambda,+}(x)|\le r_\lambda(Q)
 :=\frac{Q}{\sqrt{\lambda^2+Q^2}+\lambda}<1.
\end{equation}
On bounded subsets of $L^\infty$ this ratio depends continuously on $q$ in
the uniform norm, locally uniformly in $\lambda>0$.
\end{proposition}

\begin{proof}
Integrate \eqref{eq:Jmonotonicity} along the left Weyl solution from a
sequence tending to $-\infty$ on which that solution tends to zero.  This gives
\[
 |a(x)|^2-|b(x)|^2=2\lambda\int_{-\infty}^x|\ell(s)|^2\dd s>0.
\]
Thus $a\ne0$ and $|g_{\lambda,+}|<1$.  The quotient rule gives
\eqref{eq:riccati-plus}.  If $Q=0$, the free Weyl line gives $g=0$.
Otherwise $r=r_\lambda(Q)$ solves $Q(1-r^2)=2\lambda r$.
Where $R=|g_{\lambda,+}|>r$, its radial derivative obeys
\[
 R'\le Q(1-R^2)-2\lambda R\le-2\lambda(R-r).
\]
The disk of radius $r$ is forward invariant.  If $R(x_0)>r$, it follows
that $R(x)>r$ for every $x<x_0$ and
$R(x)-r\ge e^{2\lambda(x_0-x)}(R(x_0)-r)$, contradicting $R<1$.
This proves the bound without prescribing a boundary value at infinity.

Proposition~\ref{prop:green} gives uniform convergence of the left Weyl
lines under uniform convergence of bounded potentials.  A unit spanning
vector on such a line satisfies $|a|\ge(1+r^2)^{-1/2}$, so the map
$(a,b)\mapsto b/a$ is uniformly nonsingular.  The ratios therefore converge
uniformly.  On a compact interval of positive heights the kernel bounds and
the disk bound are uniform, giving the asserted height uniformity.
\end{proof}

\section{The forward positive density}\label{sec:density}

Fix $\kappa>0$ throughout the proof, and set
\begin{equation}\label{eq:heights}
 \lambda_a=a\kappa\quad(a=1,2,3),\qquad d=(1,-2,1).
\end{equation}
For $m\ge0$ define
\begin{equation}\label{eq:ST}
 S_m[q]=\sum_ad_ag_{\lambda_a,+}[q]^m,
 \qquad T_m[q]=\sum_ad_a\lambda_ag_{\lambda_a,+}[q]^m.
\end{equation}
The cancellations $\sum_ad_a=\sum_ad_a\lambda_a=0$ give $S_0=T_0=0$.
Set
\begin{align}
 E_\kappa[q]&=30\kappa\sum_{m\ge1}|S_m[q]|^2,
 \label{eq:positive-density}\\
 J_\kappa[q]&=60\kappa\sum_{m\ge1}
 \Ima[(2T_m[q]+qS_{m+1}[q])\overline{S_m[q]}].
 \label{eq:positive-current}
\end{align}
On $\|q\|_\infty\le Q$, the disk bound gives
$|S_m|+\kappa^{-1}|T_m|\le C_\kappa r^m$ with $r=r_\kappa(Q)<1$.
Both series therefore converge absolutely and uniformly, uniformly on each
bounded set of potentials.

For the cross-height identity put
\begin{equation}\label{eq:Cauchy-coefficients}
 a_{ab}=\frac{60d_ad_b}{a+b},\qquad
 c_a=\sum_ba_{ab},\qquad(c_1,c_2,c_3)=(5,-4,1).
\end{equation}
Define the real gauge
\begin{equation}\label{eq:forward-gauge}
 H_\kappa^+[q]=-\frac14\sum_{a,b}a_{ab}
 \log|1-g_{\lambda_a,+}\overline{g_{\lambda_b,+}}|.
\end{equation}
Every logarithm is well-defined by \eqref{eq:Schur}.

\begin{proposition}[Forward cross-height factorization]\label{prop:cross-height}
For every bounded potential,
\begin{equation}\label{eq:cross-height-identity}
 \frac12\sum_ac_a\Rea(qg_{\lambda_a,+})-\partial_xH_\kappa^+
 =E_\kappa\ge0
\end{equation}
in distributions.
\end{proposition}

\begin{proof}
Write $g_a=g_{\lambda_a,+}$ and $D_{ab}=1-g_a\overline{g_b}$.
The Riccati equation gives
\begin{equation}\label{eq:paired-plus}
 \partial_x\log D_{ab}
 =-(qg_a+\bar q\overline{g_b})
 +2(\lambda_a+\lambda_b)\frac{g_a\overline{g_b}}{D_{ab}}.
\end{equation}
Indeed, the terms involving $q,\bar q$ in $(D_{ab})_x$ equal
$-(qg_a+\bar q\overline{g_b})D_{ab}$.  Take real parts, multiply by
$a_{ab}/4$, and sum.  Symmetry, the row sums, and
$a_{ab}(\lambda_a+\lambda_b)=60\kappa d_ad_b$ yield
\[
 \frac12\sum_ac_a\Rea(qg_a)
 =\partial_xH_\kappa^+
  +30\kappa\Rea\sum_{a,b}d_ad_b\frac{g_a\overline{g_b}}{D_{ab}}.
\]
Expanding the last quotient as a geometric series gives
\[
 \Rea\sum_{a,b}d_ad_b\frac{g_a\overline{g_b}}{1-g_a\overline{g_b}}
 =\sum_{m\ge1}\left|\sum_ad_ag_a^m\right|^2.
\]
The ratios are locally Lipschitz for bounded potentials.  Their uniform
disk bound justifies the logarithmic derivatives and the absolutely
convergent series directly almost everywhere and hence in distributions.
\end{proof}

For the classical calculations, a smooth $M$ solution means a solution
belonging to $C(I;M_s)$ for every $s\ge0$, where $M_s$ is defined after
Lemma~\ref{lem:frequency-truncation}.  Its spatial derivatives are bounded
on compact time intervals.  The time Lax matrix is
\begin{equation}\label{eq:time-lax}
 V_\lambda=i\begin{pmatrix}
 2\lambda^2-|q|^2&2\lambda q+q_x\\
 2\lambda\bar q-\bar q_x&-2\lambda^2+|q|^2
 \end{pmatrix}.
\end{equation}
With $U_\lambda$ the matrix in \eqref{eq:homogeneous-lax}, direct
multiplication gives $U_t-V_x+[U,V]=0$ precisely for NLS: its $(1,2)$
entry is $q_t-iq_{xx}+2i|q|^2q$.
The fundamental matrix of $F_t=V_\lambda F$ and its inverse are bounded
uniformly in $x$ on every compact time interval.  Compatibility preserves
the spatial equation, and these bounds preserve half-line $L^2$ solutions.
Thus the time equation transports the Weyl lines.  For $g=g_{\lambda,+}$,
the quotient rule yields
\begin{equation}\label{eq:gplus-time}
 g_t=i(2\lambda\bar q-\bar q_x-4\lambda^2g+2|q|^2g
       -2\lambda qg^2-q_xg^2).
\end{equation}

\begin{proposition}[Microscopic conservation law]\label{prop:micro-law}
For every mild solution $q\in C(I;M)$,
\begin{equation}\label{eq:micro-law}
 \partial_tE_\kappa[q]+\partial_xJ_\kappa[q]=0
\end{equation}
in $\mathcal D'(I\times\R)$.  The identity is classical for smooth
solutions.  If $q_j\to q$ in $C(I_0;M)$ on a compact interval $I_0$,
then $E_\kappa[q_j]\to E_\kappa[q]$ and $J_\kappa[q_j]\to J_\kappa[q]$
uniformly on $I_0\times\R$.
\end{proposition}

\begin{proof}
First take a smooth solution.  The forward Weyl component satisfies
\[
 a_x/a=\lambda+qg,\qquad
 a_t/a=i[2\lambda^2-|q|^2+(2\lambda q+q_x)g].
\]
Commuting derivatives of these quotients gives
\begin{equation}\label{eq:forward-rho-law}
 \partial_t(qg)+\partial_xj_\lambda^+=0,
 \qquad j_\lambda^+=-i[-|q|^2+(2\lambda q+q_x)g].
\end{equation}
This computation is local and does not require a global logarithm of $a$.
Differentiating \eqref{eq:cross-height-identity} gives a current
$\frac12\sum_ac_a\Rea j_{\lambda_a}^++\partial_tH_\kappa^+$.
We check that it equals \eqref{eq:positive-current}.

Write $g_a=g_{\lambda_a,+}$, $D_{ab}=1-g_a\overline{g_b}$, and
$X_{ab}=g_a\overline{g_b}/D_{ab}$.  Symmetry gives
\[
 \partial_tH_\kappa^+
 =\frac12\Rea\sum_{a,b}a_{ab}
       \frac{(g_a)_t\overline{g_b}}{D_{ab}}.
\]
The derivative terms in the other part of the current are
$\frac12\sum_ac_a\Ima(q_xg_a)$.  Those from this display are
\[
 \frac12\Ima\sum_{a,b}a_{ab}
 \frac{\bar q_x\overline{g_b}+q_xg_a^2\overline{g_b}}{D_{ab}}
 =-\frac12\sum_ac_a\Ima(q_xg_a).
\]
To see the cancellation, interchange $a,b$ in the conjugate first term and
use $g_a^2\overline{g_b}/D_{ab}=g_a/D_{ab}-g_a$.
For the remaining terms, $X_{ba}=\overline{X_{ab}}$ implies
\[
 \Ima\sum_{a,b}a_{ab}X_{ab}=0,\qquad
 2\Ima\sum_{a,b}a_{ab}\lambda_a^2X_{ab}
 =\Ima\sum_{a,b}a_{ab}(\lambda_a^2-\lambda_b^2)X_{ab}.
\]
The terms linear in $q$ combine to
\[
 \Ima\left[q\sum_{a,b}a_{ab}(\lambda_a+\lambda_b)
                  \frac{g_a}{D_{ab}}\right].
\]
Using $a_{ab}(\lambda_a+\lambda_b)=60\kappa d_ad_b$, the current becomes
\[
 60\kappa\Ima\sum_{a,b}d_ad_b
 \left[q\frac{g_a}{D_{ab}}+(\lambda_a-\lambda_b)X_{ab}\right].
\]
The geometric expansions are
\[
 \sum_{a,b}d_ad_b\frac{g_a}{D_{ab}}
   =\sum_{m\ge1}S_{m+1}\overline{S_m},\qquad
 \sum_{a,b}d_ad_b(\lambda_a-\lambda_b)X_{ab}
   =\sum_{m\ge1}(T_m\overline{S_m}-S_m\overline{T_m}).
\]
The grade-zero term vanishes by $S_0=0$.  Taking imaginary parts proves
the stated formula and the smooth conservation law.  All rearrangements
are justified by the uniform geometric majorant.

For mild solutions, use the finite-frequency approximating solutions of
Lemma~\ref{lem:frequency-truncation} on a common local lifespan.
Uniform convergence of the ratios and the disk bound imply uniform
convergence of both series defining $E_\kappa,J_\kappa$.  Testing the
smooth law against a compactly supported test function and passing to the
limit gives the distributional law.  Iteration covers any compact subset
of the existing lifespan.  The same uniform convergence argument gives the
stated continuity for arbitrary convergent paths, not only solution paths.
\end{proof}

The spatial Riccati equation also gives
\begin{equation}\label{eq:S-spatial}
 (S_m)_x=m\bar qS_{m-1}-2mT_m-mqS_{m+1},\qquad m\ge1.
\end{equation}
For a snapshot $p\in M$ let $M_np=e^{-inx}p$.  Define
\begin{align}
 B_{n,R}[p](y)&=\int\chi_{R,y}^2 E_\kappa[M_np]\dd x,
 \label{eq:Bnr}\\
 b_{n,R}[p](y)^2&=\int\chi_{R,y}^2\sum_{m\ge1}|S_m[M_np]|^2\dd x
              =\frac{B_{n,R}[p](y)}{30\kappa},\notag\\
 c_{n,R}[p](y)^2&=\int\chi_{R,y}^2\sum_{m\ge1}|T_m[M_np]|^2\dd x.
 \notag
\end{align}
Define their spatial suprema 
\begin{equation}\label{eq:beta-gamma}
 \beta_{n,R}(p)=\sup_y b_{n,R}[p](y),\qquad
 \gamma_{n,R}(p)=\sup_y c_{n,R}[p](y).
\end{equation}
The quadratic part of $B$ is
\begin{equation}\label{eq:Bnr-quadratic}
 B_{n,R}^{[2]}[p](y)=30\kappa
 \|\chi_{R,y}h_\kappa(D)\overline{M_np}\|_2^2,
 \qquad
 h_\kappa(s)=\frac1{2\kappa+is}-\frac2{4\kappa+is}+\frac1{6\kappa+is}.
\end{equation}
Indeed, this equation follows by linearizing the Riccati equation; higher grades and nonlinear
corrections contribute $O(\varepsilon^4)$ to $B[\varepsilon p]$,
uniformly in $y$ for fixed $p,n,R$.

Along a solution use $B_{n,R}(t,y),b_{n,R}(t,y),c_{n,R}(t,y)$ for the
same expressions with $M_np$ replaced by the full transform $q_n=\cG_nq$.
Constant-phase covariance of the ratios gives
\begin{equation}\label{eq:static-dynamic-window}
 B_{n,R}(t,y)=B_{n,R}[q(t)](y+2nt),
\end{equation}
and the analogous identities for $b,c$.  Thus their suprema are snapshot
functionals.  The conservation law yields
\begin{equation}\label{eq:B-flux}
 \partial_tB_{n,R}(t,y)
 =\int(\chi_{R,y}^2)'J_\kappa[q_n](t,x)\dd x.
\end{equation}
For mild solutions this is initially distributional in time;
Section~\ref{sec:active-flux} proves its uniform continuous-time version.

Use the calibration
\begin{equation}\label{eq:saturated-calibration}
 \widehat\Theta_\kappa(B)=\kappa\sqrt{\frac B{\kappa+B}},\qquad B\ge0,
\end{equation}
and define
\begin{equation}\label{eq:KR}
 K_R(p)=\sum_n\sup_y\widehat\Theta_\kappa(B_{n,R}[p](y)),
 \qquad k_R(p)=K_R(p)/\kappa.
\end{equation}
The sum is initially allowed to be infinite.  Since the calibration is
increasing,
\begin{equation}\label{eq:exact-summand}
 k_R=\sum_n\frac{\sqrt{30}\,\beta_{n,R}}{\sqrt{1+30\beta_{n,R}^2}},
\end{equation}
and hence
\begin{equation}\label{eq:k-beta}
 k_R\asymp\sum_n\min(\beta_{n,R},1).
\end{equation}

\begin{lemma}[Raw localized flux bound]\label{lem:raw-flux-bound}
For every $p\in M$, $n\in\Z$, $R\ge1$, and $y\in\R$,
\begin{equation}\label{eq:raw-flux-bound}
 \left|\int(\chi_{R,y}^2)'J_\kappa[M_np]\dd x\right|
 \le\frac{C_\chi\kappa}{R}b_{n,R}[p](y)
       \big(c_{n,R}[p](y)+\|p\|_\infty b_{n,R}[p](y)\big).
\end{equation}
The same bound holds for the dynamic boxes along every mild solution.
\end{lemma}

\begin{proof}
The left shift is a contraction on $\ell^2$, so pointwise
\[
 |J_\kappa[M_np]|
 \le C\kappa\big(|T|_{\ell^2}|S|_{\ell^2}
                    +\|p\|_\infty|S|_{\ell^2}^2\big).
\]
Use $|(\chi_{R,y}^2)'|\le C_\chi R^{-1}\chi_{R,y}^2$ and
Cauchy--Schwarz in $x$.  The dynamic statement follows by translation and
constant-phase covariance.  No differentiability of the potential is used.
\end{proof}

\section{An exact nonlinear spectral frame}\label{sec:fock}

The powers of the Riccati ratio admit a self-adjoint linearization on a
graded sequence space. We call this structure a Fock realization and introduce it in this section. It gives an elegant way of estimating in later sections. 

Let $\mathsf Ne_m=me_m$ on $\ell^2(\mathbb N)$, put $D=-i\partial_x$,
and set $A_0=D\mathsf N^{-1}$ on $\mathcal H=L^2(\R;\ell^2)$.
With $X_0=0$, define
\begin{equation}\label{eq:Aplus}
 (A_q^+X)_m=\frac DmX_m+i\bar qX_{m-1}-iqX_{m+1}.
\end{equation}
The free operator is self-adjoint on
\[
 \operatorname{Dom}(A_0)
 =\left\{X\in\mathcal H:\sum_{m\ge1}m^{-2}\|DX_m\|_2^2<\infty\right\}.
\]
The two shifts are adjoints of each other, and their sum has norm at most
$2\|q\|_\infty$.  Thus $A_q^+$ is self-adjoint on the same domain.
Its velocity operator $\mathsf N^{-1}$ is a positive contraction, so
Lemmas~\ref{lem:ulocfc}--\ref{lem:uloc-resolvent-factorization} apply.

The two rational filters are
\begin{equation}\label{eq:h-t-filters}
 h_\kappa(s)=\frac{8\kappa^2}
 {(2\kappa+is)(4\kappa+is)(6\kappa+is)},
 \qquad t_\kappa(s)=-\frac{is}{2}h_\kappa(s).
\end{equation}
The partial fractions of $h_\kappa$ are those in
\eqref{eq:Bnr-quadratic}.  The numerator $8\kappa^2$ is the remainder
after cancellation of the constant and affine height moments.

\begin{lemma}[Self-adjoint realization of the features]\label{lem:fock-realization}
For $p\in M$ and $\lambda>0$, in $L^2_{\uloc}(\R;\ell^2)$,
\begin{equation}\label{eq:weyl-resolvent-plus}
 V_{\lambda,+}[p]:=(g_{\lambda,+}[p]^m)_{m\ge1}
 =(2\lambda+iA_p^+)^{-1}(\bar p e_1).
\end{equation}
Consequently
\begin{equation}\label{eq:determinant-fock-plus}
 (S_m[p])_{m\ge1}=h_\kappa(A_p^+)(\bar p e_1),\qquad
 (T_m[p])_{m\ge1}=t_\kappa(A_p^+)(\bar p e_1).
\end{equation}
\end{lemma}

\begin{proof}
The Schur bound gives
\[
 \sum_{m\ge1}|g_{\lambda,+}|^{2m}
 \le\frac{r_\lambda(\|p\|_\infty)^2}
          {1-r_\lambda(\|p\|_\infty)^2}
 \le\frac{\|p\|_\infty}{2\lambda}.
\]
Thus $V_{\lambda,+}\in L^2_{\uloc}$.  Multiplication of the Riccati
equation by $g^{m-1}$ gives, componentwise,
\begin{equation}\label{eq:moment-plus}
 (2\lambda+iA_p^+)V_{\lambda,+}=\bar p e_1.
\end{equation}
For $m\ge2$ the neighboring powers cancel the shift terms; at $m=1$ the
missing grade zero leaves the source.  Uniformly local uniqueness in
Lemma~\ref{lem:uloc-resolvent-factorization} identifies the solution with
the displayed resolvent.  Sum over the three heights and use
\[
 \sum_a\frac{d_a}{2\lambda_a+is}=h_\kappa(s),\qquad
 \sum_a\frac{d_a\lambda_a}{2\lambda_a+is}
 =-\frac{is}{2}h_\kappa(s),
\]
where the second identity uses $\sum_ad_a=0$.
\end{proof}

The unitary gauge $(\mathcal U_nX)_m=e^{imnx}X_m$ gives
\begin{equation}\label{eq:covariance-plus}
 \mathcal U_n^{-1}A_{M_nq}^+\mathcal U_n=A_q^++n,
 \qquad \mathcal U_n^{-1}(\overline{M_nq}\,e_1)=\bar q e_1.
\end{equation}
Indeed, differentiating the grade-$m$ phase gives $mn$, which becomes
$n$ after division by $m$, and the phases in both shifts cancel.
Define the gauged features
\begin{equation}\label{eq:features-plus}
 F_n^+=h_\kappa(A_q^++n)(\bar q e_1),\qquad
 G_n^+=t_\kappa(A_q^++n)(\bar q e_1).
\end{equation}
Their norms are
\begin{equation}\label{eq:b-feature}
 \beta_{n,R}=\|F_n^+\|_{U_R},\qquad
 \gamma_{n,R}=\|G_n^+\|_{U_R}.
\end{equation}
For dynamic windows the centers are translated by $2nt$ as in
\eqref{eq:static-dynamic-window}; the spatial suprema are unchanged.

\begin{proposition}[Feature summability]\label{prop:feature-summability}
For every $q\in M$ and fixed $R\ge1$,
\begin{equation}\label{eq:feature-summability}
 \sum_n(\beta_{n,R}+\gamma_{n,R})<\infty.
\end{equation}
In particular, $K_R(q)<\infty$.
\end{proposition}

\begin{proof}
It suffices to take $R=1$ by Lemma~\ref{lem:windows}.  Put
$a_j=\|\square_jq\|_\infty$, $N=\sum_ja_j$, and
\begin{equation}\label{eq:tail-envelope}
 A_n=\sum_j\frac{a_j}{\kappa+|j-n|}.
\end{equation}
The case $N=0$ is immediate.  For $p_n=M_nq$ let
$\mathcal R_\lambda=(2\lambda+\partial_x)^{-1}$, whose one-sided
convolution kernel has $L^1$ norm $(2\lambda)^{-1}$.  Thus, at the three
fixed heights,
\begin{equation}\label{eq:riccati-resolvent-bounds}
 \|\mathcal R_\lambda f\|_M\le C_\kappa\|f\|_M,
 \qquad \|\mathcal R_\lambda\overline{p_n}\|_M\le C_\kappa A_n.
\end{equation}
For the second bound, the localized multiplier on the conjugate of the
$j$th input box has size $O_\kappa((\kappa+|j-n|)^{-1})$, as do its first
two derivatives.  The compactly supported multiplier estimate
$\|\check a\|_1\lesssim\|a\|_1+\|a''\|_1$ gives the same bound for
its kernel.  Summing the boxes proves the assertion.

When $NA_n\le c_\kappa$, the Riccati map
\[
 u\longmapsto\mathcal R_\lambda(\overline{p_n}-p_nu^2)
\]
is a contraction on a ball of radius $C_\kappa A_n$ in $M$.
Indeed, the map and its difference are bounded by
$C_\kappa A_n+C_\kappa N\|u\|_M^2$ and
$C_\kappa N(\|u\|_M+\|v\|_M)\|u-v\|_M$, respectively.
Since $A_n\le N/\kappa$ and $A_n\le c_\kappa/N$, we also have
$A_n^2\le c_\kappa/\kappa$; decrease $c_\kappa$ to make the ball
uniformly small.  The fixed point is the Weyl ratio: solve
$a_x=(\lambda+p_nu)a$, $b=ua$.  The good-index condition gives
$\|p_nu\|_\infty\le\lambda/2$, so $(a,b)$ is square integrable at
$-\infty$.  Uniqueness of the Weyl line identifies $u$.
It follows that
\begin{equation}\label{eq:riccati-error-bounds}
 \|g_{\lambda,+}[p_n]\|_M\le C_\kappa A_n,
 \qquad \|g_{\lambda,+}[p_n]-\mathcal R_\lambda\overline{p_n}\|_M
 \le C_\kappa NA_n^2.
\end{equation}

The linear approximations to $S_1$ and $T_1$ are obtained by summing
$\mathcal R_{\lambda_a}\overline{p_n}$ with coefficients $d_a$ and
$d_a\lambda_a$, respectively.  Their Fourier multipliers are
\[
 \sum_a\frac{d_a}{2\lambda_a+i\xi}
 =\frac{8\kappa^2}{(2\kappa+i\xi)(4\kappa+i\xi)(6\kappa+i\xi)},
 \qquad
 \sum_a\frac{d_a\lambda_a}{2\lambda_a+i\xi}
 =-\frac{i\xi}{2}\sum_a\frac{d_a}{2\lambda_a+i\xi}.
\]
Thus a frequency box centered at $j$ contributes with weights bounded
by $C_\kappa\langle j-n\rangle^{-3}$ for $S_1$ and
$C_\kappa\langle j-n\rangle^{-2}$ for $T_1$.  Applying the same
localized multiplier estimate and adding the nonlinear remainder gives
\begin{align}
 \|S_1[p_n]\|_{U_1}&\le C_\kappa(k_3*a)_n+C_\kappa NA_n^2,
 \label{eq:first-grade-feature}\\
 \|T_1[p_n]\|_{U_1}&\le C_\kappa(k_2*a)_n+C_\kappa NA_n^2,
 \label{eq:first-grade-companion}
\end{align}
where $k_r(j)=\langle j\rangle^{-r}$.
The algebra estimate gives higher-grade bounds
\[
 \|S_m[p_n]\|_M+\|T_m[p_n]\|_M\le C_\kappa(C_\kappa A_n)^m.
\]
Choose the good threshold so that $C_\kappa A_n\le1/2$.  Then
\begin{equation}\label{eq:higher-grade-feature}
 \left(\sum_{m\ge2}
   (\|S_m[p_n]\|_{U_1}^2+\|T_m[p_n]\|_{U_1}^2)\right)^{1/2}
 \le C_\kappa A_n^2.
\end{equation}
Combining these estimates gives, at every good index,
\begin{equation}\label{eq:good-feature-envelope}
 \beta_{n,1}+\gamma_{n,1}
 \le C_\kappa(k_2*a)_n+C_\kappa(1+N)A_n^2.
\end{equation}
Discrete Young's inequality gives $\|A\|_{\ell^2}\le C_\kappa N$.
Thus the exceptional set $\{n:NA_n>c_\kappa\}$ has at most
$C_\kappa N^4$ elements.  On that set the disk bound directly gives
$\beta_{n,1}+\gamma_{n,1}\le C_{\kappa,\chi}\sqrt N$.
Summing there and using \eqref{eq:good-feature-envelope} on its complement
proves the result.  No uniform smallness of $q$ is assumed.
\end{proof}

\begin{proposition}[Continuity and uniform tails of the features]
\label{prop:feature-continuity}
For every fixed $R\ge1$, the map
\begin{equation}\label{eq:feature-l1-continuity}
 q\longmapsto
 \bigl((\beta_{n,R}(q))_n,(\gamma_{n,R}(q))_n\bigr)
\end{equation}
is continuous from $M$ to $\ell^1(\Z)\times\ell^1(\Z)$.  In particular,
$K_R:M\to[0,\infty)$ is continuous.  If $q(t)\in C(I;M)$ and $I_0$ is a
compact subinterval of $I$, then the $\ell^1$ tails of both feature
sequences tend to zero uniformly for $t\in I_0$.
\end{proposition}

\begin{proof}
For a fixed $n$, continuity in $U_R$ follows from the resolvent identity
for each of the three factors in $h_\kappa(A_q^++n)$.  The multiplier $t_\kappa$ is a finite linear combination of the same
resolvents by the identities in the proof of
Lemma~\ref{lem:fock-realization}.  The source converges in $U_R$, multiplication
by the difference of two potentials has norm controlled by their
$L^\infty$ distance, and Lemma~\ref{lem:ulocfc} bounds every remaining
resolvent.  Taking a uniformly local norm and then a spatial supremum shows
that each $\beta_{n,R}$ and $\gamma_{n,R}$ is continuous.  Concretely,
\[
 |\beta_{n,R}(q')-\beta_{n,R}(q)|
 \le \|F_n^+(q')-F_n^+(q)\|_{U_R},
\]
and the analogous inequality holds for $\gamma$.  Thus the supremum over all
translations $y$ is already part of the norm estimate and does not require
pointwise convergence uniformized afterward.

It remains to make this uniform in the tail.  Let $q'$ range in a small
$M$-neighborhood of $q$, and use primes for the sequences in
\eqref{eq:tail-envelope}.  There is a common bound $N_*$.  Moreover,
\begin{align*}
 \norm{k_2*(a'-a)}_{\ell^1}&\le C\norm{q'-q}_M,\\
 \norm{A'-A}_{\ell^2}&\le C_\kappa\norm{q'-q}_M.
\end{align*}
Choose a finite set $F$ so that the $\ell^1$ tail of $k_2*a$, the $\ell^2$
tail of $A$, and $N_*\sup_{n\notin F}A_n$ are all sufficiently small.
Since $\norm{A'-A}_{\ell^\infty}\le\norm{A'-A}_{\ell^2}$, after shrinking
the neighborhood every index outside $F$ obeys
the good condition $N_*A_n'\le c_\kappa$.  The estimate
\eqref{eq:good-feature-envelope}, together with
\[
 \norm{(A')^2-A^2}_{\ell^1}
 \le\norm{A'-A}_{\ell^2}(\norm{A'}_{\ell^2}+\norm A_{\ell^2}),
\]
now makes the feature tails uniformly small.  Continuity on the finite set
$F$ proves \eqref{eq:feature-l1-continuity} for $R=1$, and
Lemma~\ref{lem:windows} proves it for fixed $R$.

By \eqref{eq:exact-summand}, the scalar calibration is a Lipschitz
function of $\beta_{n,R}$, uniformly in $n$; hence $K_R$ is continuous.
Finally, the continuous image of a compact time interval in $\ell^1$ is
compact.  To see that a compact subset has uniformly vanishing tails, cover
it by a finite $\varepsilon$-net and choose one finite tail cutoff that works
for every center of the net.
\end{proof}

Choose once and for all a real nonnegative
$\theta\in C_c^\infty(\R)$ with $\sum_n\theta(s+n)=1$; the function
$\varphi$ in Section~\ref{sec:main} is one choice.  Since $h_\kappa$ has
no real zeros, define
\begin{equation}\label{eq:rn}
 r_n(s)=\frac{\theta(s+n)}{h_\kappa(s+n)},\qquad
 \phi_n(s)=r_n(s)h_\kappa(s+n)=\theta(s+n).
\end{equation}
These symbols are translates of fixed smooth compactly supported functions.
For a Fourier--Stieltjes multiplier
$a(s)=\int e^{-its}\dd\mu_a(t)$, write
\[
 \|a\|_{\mathcal W_\sigma}
 =\int(1+|t|)^\sigma\dd|\mu_a|(t).
\]
For the smooth symbols here,
$\dd\mu_a(t)=(2\pi)^{-1}\widehat a(-t)\dd t$.

\begin{lemma}[Localized spectral multipliers]\label{lem:rational-overlap}
Uniformly in $n,k,\ell\in\Z$,
\begin{align}
 \|(\cdot+n)^r r_n\|_{\mathcal W_1}&\le C_\kappa,
       \qquad r=0,1,2,\label{eq:rn-wiener}\\
 \|h_\kappa(\cdot+k)\phi_n\|_{\mathcal W_{1/2}}
       &\le C_\kappa\langle n-k\rangle^{-3},\label{eq:frame-overlap}\\
 \|t_\kappa(\cdot+n)r_\ell\|_{\mathcal W_1}
       &\le C_\kappa\langle n-\ell\rangle^{-2}.
       \label{eq:companion-wiener}
\end{align}
\end{lemma}

\begin{proof}
For a smooth symbol supported in a fixed compact interval, Fourier
inversion and four integrations by parts give
\[
 \|a\|_{\mathcal W_1}
 \lesssim\sum_{j=0}^4\|a^{(j)}\|_1.
\]
Translation of the symbol only modulates its Fourier transform.  The first
bound therefore follows by translating $s+n$ to zero.  On the support of
$\phi_n$, $s+n$ is bounded and every derivative of $h_\kappa(s+k)$ is
$O_\kappa(\langle k-n\rangle^{-3})$.  On the support of $r_\ell$, every
derivative of $t_\kappa(s+n)$ is
$O_\kappa(\langle n-\ell\rangle^{-2})$.  Apply the displayed estimate
to the translated products.  This also proves all fixed polynomial-moment
variants used below.
\end{proof}

\begin{proposition}[Endpoint dual frame]\label{prop:dual-frame}
For every $q\in M$,
\begin{equation}\label{eq:dual-frame-plus}
 \bar q e_1=\sum_n r_n(A_q^+)F_n^+
\end{equation}
with absolute convergence in $L^2_{\uloc}(\R;\ell^2)$.
\end{proposition}

\begin{proof}
Put $A=A_q^+$ and $g=\bar q e_1$.  Feature summability and
\eqref{eq:rn-wiener} give an absolutely convergent sum $W$ on the right.
Products of the Fourier-measure multipliers agree with products of their
symbols on uniformly local data: apply Fubini to the group integrals and
use finite propagation and the integrable measure tails to pass through
spatial exhaustion.  Thus the finite partial sums are $\Phi_N(A)g$, where
$\Phi_N=\sum_{|n|\le N}\phi_n$.
For fixed $k$, the overlap bound and $\sum_n\phi_n=1$ give
\[
 \|h_\kappa(A+k)(g-\Phi_N(A)g)\|_{U_1}
 \le C_\kappa\|g\|_{U_1}
       \sum_{|n|>N}\langle n-k\rangle^{-3}\longrightarrow0.
\]
The scalar identity underlying this bound converges in
$\mathcal W_{1/2}$, so its use on uniformly local data is justified by
Lemma~\ref{lem:ulocfc}.  Passing to the limit gives
$h_\kappa(A+k)(g-W)=0$.  Injectivity from
Lemma~\ref{lem:uloc-resolvent-factorization} implies $W=g$.
\end{proof}

\section{Weighted resolvent observation and coercivity}
\label{sec:observation-coercivity}

The positive principal velocities of the graded operator give a stationary
estimate that controls the first component without any amplitude loss.

\begin{lemma}[Weighted resolvent observation]\label{lem:one-way}
Let $A=-i\mathsf N^{-1}\partial_x+W(x)$ with $W=W^*$ bounded and
strongly measurable.  For $a>0$, $c\in\R$, and $f\in L^2_{\uloc}$,
\begin{equation}\label{eq:spectral-observation}
 \|((a\pm i(A+c))^{-1}f)_1\|_\infty
 \le C_{a,\chi}\|f\|_{U_1}.
\end{equation}
The constant is independent of $W$ and $c$.
\end{lemma}

\begin{proof}
Let $V=\mathsf N^{-1}$ and $u=(a+i(A+c))^{-1}f$, using the uniformly
local resolvent.  Its equation is
$Vu_x+au+i(W+c)u=f$.  The stationary energy identity is
\[
 j'+2a|u|^2=2\Rea\langle u,f\rangle,
 \qquad j=\langle Vu,u\rangle,\qquad |u_1|^2\le j\le|u|^2.
\]
This holds distributionally for the graph-domain representatives.  One
justification is to localize $u$ with a compact spatial cutoff and use the
graph approximation from Lemma~\ref{lem:ulocfc}: $Vu_x$ is locally in
$L^2$ by the equation, and the bounded self-adjoint zeroth-order term
cancels in the real part.  The resulting identity shows that $j$ has a
locally absolutely continuous representative.  The first component of the
equation also gives $u_1\in H^1_{\mathrm{loc}}$.

Fix $0<\delta<2a$.  Integrate with weight $e^{\delta(y-x)}$ on
$(-\infty,x]$.  The weighted boundary term at $-\infty$ vanishes along
a sequence because $u\in L^2_{\uloc}$; all integrals converge by the
same exponential weight.  Since $0\le j\le|u|^2$,
\[
 j(x)+(2a-\delta)\int_{-\infty}^x e^{\delta(y-x)}|u(y)|^2\dd y
 \le2\Rea\int_{-\infty}^x
                 e^{\delta(y-x)}\langle u(y),f(y)\rangle\dd y.
\]
Young's inequality gives
\[
 |u_1(x)|^2\le j(x)
 \le\frac1{2a-\delta}\int_{-\infty}^x
                          e^{\delta(y-x)}|f(y)|^2\dd y.
\]
Choose $\delta=a$ and sum the exponentially weighted unit intervals.
Lemma~\ref{lem:windows} bounds the result by $C_{a,\chi}\|f\|_{U_1}^2$.
For the minus resolvent the energy equation is
$-j'+2a|u|^2=2\Rea\langle u,f\rangle$; integrate instead on
$[x,\infty)$ with weight $e^{-\delta(y-x)}$.
\end{proof}

\begin{proposition}[Nonlinear Bernstein estimate]\label{prop:observation}
For every $q\in M$, $n\in\Z$, and $R\ge1$,
\begin{equation}\label{eq:nonlinear-Bernstein}
 \|(r_n(A_q^+)F_n^+)_1\|_\infty\le C_{\kappa,\chi}\beta_{n,R}.
\end{equation}
\end{proposition}

\begin{proof}
Fix $a=2\kappa$ and put
$b(s)=(a+is)\theta(s)/h_\kappa(s)$.  This is one fixed smooth compactly
supported symbol, and
\[
 r_n(A)=(a+i(A+n))^{-1}b(A+n),\qquad A=A_q^+.
\]
Lemma~\ref{lem:ulocfc} bounds $b(A+n)$ on $U_1$ uniformly in $q,n$.
Apply Lemma~\ref{lem:one-way} and use
$\|F_n^+\|_{U_1}\le\|F_n^+\|_{U_R}=\beta_{n,R}$.
Feature summability then gives absolute convergence of the first
components of the dual frame in $L^\infty$; their sum agrees with the
uniformly local reconstruction.
\end{proof}

\begin{corollary}[Amplitude observability]\label{cor:amplitude}
Let $Q=\norm q_\infty$ and $\cL_R=\sum_n\beta_{n,R}$.  Then
\begin{equation}\label{eq:Q-L}
 Q\le C_{\kappa,\chi}\cL_R.
\end{equation}
Furthermore,
\begin{equation}\label{eq:Q-k}
 Q\le C_{\kappa,\chi}(k_1+k_1^2)
 \le C_{\kappa,\chi}(k_R+k_R^2).
\end{equation}
\end{corollary}

\begin{proof}
Take the first component in \eqref{eq:dual-frame-plus}, use
Proposition~\ref{prop:observation}, and sum.  This proves \eqref{eq:Q-L}.
The Schur estimate and the geometric series in
\eqref{eq:positive-density} give
\begin{equation}\label{eq:beta-cap}
 \sup_n\beta_{n,1}^2\le C_{\kappa,\chi}Q.
\end{equation}
Split $\cL_1$ into boxes with $\beta_{n,1}\le1$ and
$\beta_{n,1}>1$.  Equation \eqref{eq:k-beta} and
\eqref{eq:beta-cap} give
\[
 \cL_1\le C(1+\sqrt Q)k_1.
\]
Combine this with \eqref{eq:Q-L} and solve the quadratic inequality in
$x=\sqrt Q$: from $x^2\le Ck_1(1+x)$ one obtains
$x\le C(\sqrt{k_1}+k_1)$ and hence $Q\le C(k_1+k_1^2)$.  This proves the
first estimate in \eqref{eq:Q-k}.  Since the
windows are unnormalized and decreasing, $\chi_{R,y}\ge\chi_{1,y}$ for
$R\ge1$, so $k_1\le k_R$.
\end{proof}

We now prove that the saturated functional controls the complete modulation
norm.  The first-component estimate alone controls only $L^\infty$; the
following argument recovers the outer $\ell^1$ frequency summation.

\begin{lemma}[Uniformly local box calculus]\label{lem:local-box-calculus}
There are fixed constants $C,C_0$ with the following properties.  Let
$c\in\R$ and $j,k\in\Z$.
\begin{enumerate}[label=\textup{(\roman*)}]
\item If $f,F\in L^2_{\uloc}$, $(D+c)f=F$ in distributions, and
$|j+c|\ge4$, then
\begin{equation}\label{eq:off-box}
 \sup_y\norm{\chi_{1,y}\square_jf}_2
 \le\frac C{|j+c|}
 \sup_y\norm{\widetilde\chi_yF}_2.
\end{equation}
For $|j+c|<4$ one has instead the direct bound
$\sup_y\norm{\chi_{1,y}\square_jf}_2\le C\norm f_{U_1}$.
\item If $q\in M$, $f\in L^2_{\uloc}$, and
\[
 \alpha_p=\sup_y\norm{\chi_{1,y}\square_pq}_2,
\]
then
\begin{equation}\label{eq:local-product}
 \norm{\square_k(qf)}_{L^2_{\uloc}}
 \le C\sum_p\alpha_p
 \sum_{|j-(k-p)|\le C_0}
 \norm{\square_jf}_{L^2_{\uloc}}.
\end{equation}
The same estimates hold with any of the finitely many enlarged windows
which arise from the fixed frequency cutoffs.
\end{enumerate}
\end{lemma}

\begin{proof}
Put $J_y=[y-1,y+1]$, so that $\norm F_{I_1}
=\sup_y\norm{\one_{J_y}F}_2$.
In the off-resonant case,
\[
 \square_jf
 =\mathcal F^{-1}\left[
 \frac{\varphi(\xi-j)}{\xi+c}\widehat F(\xi)\right].
\]
This follows directly in distributions from
$(\xi+c)\widehat f=\widehat F$: on the support of
$\varphi(\xi-j)$ the factor $\xi+c$ is nonzero, so multiplication by its
smooth reciprocal is legitimate.  A homogeneous resonant distribution is
supported at $\xi=-c$ and is killed by this off-resonant cutoff.  In the
central boxes we do not invert $D+c$ at all; we simply use
$\square_jf=\check\varphi_j*f$ and its uniform $L^1$ kernel bound.
After writing $\xi=j+\eta$, the denominator is $j+c+\eta$.  Repeated
integration by parts shows that the convolution kernel $K_{j,c}$ of this
multiplier satisfies
\[
 \norm{K_{j,c}}_{L^1}
 \le C|j+c|^{-1}
\]
uniformly when $|j+c|\ge4$.  Convolution is bounded on the interval
uniformly local norm without any enlargement: Minkowski's inequality gives
\[
 \norm{\one_{J_y}(K*F)}_2
 \le\int_\R|K(z)|\norm{\one_{J_y-z}F}_2\dd z
 \le\norm K_1\norm F_{I_1}.
\]
Taking the spatial supremum and using Lemma~\ref{lem:windows} to pass
between $I_1$, $U_1$, and the finitely enlarged fixed windows proves
\eqref{eq:off-box}.  In the central boxes the projection kernel itself has
a uniform $L^1$ norm, which gives the direct estimate on $I_1$.

For the product, choose a fixed $\vartheta\in C_c^\infty$ which equals one on
$\supp\varphi$ and put $\widetilde\square_p=\vartheta(D-p)$.  Then
$\square_pq=\widetilde\square_p\square_pq$, and the convolution-kernel
argument used in \eqref{eq:modulation-uloc}, with a fixed enlargement of
$J_y$, gives the precise local Bernstein estimate
\[
 \norm{\square_pq}_\infty
 \le C\sup_y\norm{\one_{[y-C_0,y+C_0]}\square_pq}_2
 \le C\alpha_p.
\]
Now use these fattened projections to write
\[
 \square_k(qf)=
 \sum_{\substack{p,j\in\Z\\|k-p-j|\le C_0}}
 \square_k\bigl((\square_pq)(\square_jf)\bigr).
\]
It follows directly that
\[
 \norm{(\square_pq)(\square_jf)}_{I_1}
 \le C\alpha_p\norm{\square_jf}_{I_1}.
\]
The output projection is convolution with a kernel having a uniform
$L^1$ norm, hence is bounded on $I_1$ by the first calculation in this
proof.  Take the spatial supremum, sum the finitely overlapping frequency
indices, and transfer back to the fixed-window norm by
Lemma~\ref{lem:windows}.  This proves \eqref{eq:local-product}.
\end{proof}

The preceding box calculus upgrades amplitude observability to control of the
outer \(\ell^1\) frequency sum.  Combining the two now yields the coercive
estimate used in the endpoint continuation argument.

\begin{theorem}[Coercivity of the saturated determinant]\label{thm:coercivity}
For every $q\in M$ and $R\ge1$,
\begin{equation}\label{eq:saturated-coercivity}
 \norm q_M
 \le C_{\kappa,\chi}
 k_R(q)(1+k_R(q))^9.
\end{equation}
In particular, $K_R(q)=0$ if and only if $q=0$.
\end{theorem}

\begin{proof}
It is enough to prove the estimate for $R=1$.  Put
\[
 \alpha_p=\sup_y\norm{\chi_{1,y}\square_pq}_2,
 \qquad
 N=\sum_p\alpha_p\asymp_\chi\norm q_M,
 \qquad
 L=\sum_\ell\beta_{\ell,1},
 \qquad Q=\norm q_\infty.
\]
The compact frequency support of each $\square_pq$, together with the
uniform $L^1$ bounds for the projection kernels, gives
\begin{equation}\label{eq:alpha-cap}
 \sup_p\alpha_p\le C_\chi Q.
\end{equation}

Let $g=\bar q e_1$ and define the forward frame pieces
\begin{equation}\label{eq:u-ell}
 u_\ell=\theta(A_q^++\ell)g,
 \qquad
 u_\ell^{(r)}=(A_q^++\ell)^ru_\ell,\quad r=0,1,2.
\end{equation}
Then $g=\sum_\ell u_\ell$ by the scalar frame identity.  Moreover,
\[
 u_\ell=r_\ell(A_q^+)F_\ell^+,\qquad
 u_\ell^{(r)}=[(s+\ell)^r r_\ell(s)](A_q^+)F_\ell^+.
\]
The multipliers in brackets are translates of fixed smooth compactly
supported symbols.  Lemmas~\ref{lem:rational-overlap} and
\ref{lem:ulocfc} give
\begin{equation}\label{eq:adapted-bounds}
\sum_{r=0}^2\norm{u_\ell^{(r)}}_{U_1}
 \le C_{\kappa,\chi}\beta_{\ell,1}.
\end{equation}
The identities defining these powers hold in distributions on uniformly
local data.  Indeed, they hold on spatially truncated $L^2$ sources; the
localized symbols for successive powers converge locally in $L^2$, and
$A_q^+$ maps $L^2_{\mathrm{loc}}$ continuously into distributions.
Passing to the limit one factor at a time justifies the component equations
below without asserting membership in a global $L^2$ operator domain.

If $X=u_\ell^{(r)}$ and $Y=u_\ell^{(r+1)}$, the $m$th component of
$Y=(A_q^++\ell)X$ is
\begin{equation}\label{eq:component-equation}
 \left(\frac Dm+\ell\right)X_m
 =Y_m-i\bar qX_{m-1}+iqX_{m+1}.
\end{equation}
We shall use the two estimates in Lemma~\ref{lem:local-box-calculus}.

Apply \eqref{eq:component-equation} with $m=2$, first for $r=0$ and then
for $r=1$.  After multiplication by two, the equation is
\[
 (D+2\ell)(u_\ell^{(r)})_2
 =2(u_\ell^{(r+1)})_2
  -2i\bar q(u_\ell^{(r)})_1
  +2iq(u_\ell^{(r)})_3.
\]
Its right side has $U_1$ norm at most
$C(1+Q)\beta_{\ell,1}$ by \eqref{eq:adapted-bounds}.  Apply
\eqref{eq:off-box}; in the central boxes use the direct part of
Lemma~\ref{lem:local-box-calculus}.  We obtain
\begin{equation}\label{eq:second-component}
 \sup_y\norm{\chi_{1,y}\square_j(u_\ell^{(r)})_2}_2
 \le C_{\kappa,\chi}
 \frac{(1+Q)\beta_{\ell,1}}
      {\langle j+2\ell\rangle},
 \qquad r=0,1.
\end{equation}
Here the finitely many boxes with $j=-2\ell+O(1)$ are bounded directly by
\eqref{eq:adapted-bounds}.  On those boxes
$\langle j+2\ell\rangle\le C$, so the direct bound
$C\beta_{\ell,1}$ is at most the right side of
\eqref{eq:second-component}.  Off the central set,
$|j+2\ell|^{-1}\lesssim\langle j+2\ell\rangle^{-1}$.
This combines the two cases into the single Japanese-bracket estimate.

Next apply \eqref{eq:component-equation} twice with $m=1$.  The first
component satisfies the two exact equations
\begin{align*}
 (D+\ell)(u_\ell)_1
 &=(u_\ell^{(1)})_1+iq(u_\ell)_2,\\
 (D+\ell)(u_\ell^{(1)})_1
 &=(u_\ell^{(2)})_1+iq(u_\ell^{(1)})_2.
\end{align*}
Fix an output box $-n$.  When $|\ell-n|$ is larger than the fixed overlap
constant, $(D+\ell)^{-1}\square_{-n}$ denotes the bounded multiplier
$\varphi(\xi+n)/(\xi+\ell)$, and similarly for the second power.  Applying
$\square_{-n}$ to the preceding two equations gives
\begin{align*}
 \square_{-n}(u_\ell)_1
={}&(D+\ell)^{-2}\square_{-n}(u_\ell^{(2)})_1
 +i(D+\ell)^{-2}\square_{-n}\bigl(q(u_\ell^{(1)})_2\bigr)\\
 &+i(D+\ell)^{-1}\square_{-n}\bigl(q(u_\ell)_2\bigr).
\end{align*}
The same bounds follow from \eqref{eq:off-box}.  In the finitely many
boxes with $\ell-n=O(1)$, use \eqref{eq:adapted-bounds}; the first term on
the right below then dominates the direct estimate.  This yields
\begin{align}
 &\sup_y\norm{\chi_{1,y}\square_{-n}(u_\ell)_1}_2
 \notag\\
 &\quad\le C\frac{\beta_{\ell,1}}{\langle\ell-n\rangle^2}
 +\frac C{\langle\ell-n\rangle}
  \norm{\square_{-n}(q(u_\ell)_2)}_{L^2_{\uloc}}
 \notag\\
 &\qquad
 +\frac C{\langle\ell-n\rangle^2}
  \norm{\square_{-n}(q(u_\ell^{(1)})_2)}_{L^2_{\uloc}}.
 \label{eq:first-component-two-inversions}
\end{align}
To estimate the two products, \eqref{eq:local-product} restricts the
$p$th box of $q$ to boxes
\[
 j=-n-p+O(1)
\]
of $(u_\ell^{(r)})_2$.  At those boxes the denominator in
\eqref{eq:second-component} is
$\langle j+2\ell\rangle\asymp
 \langle2\ell-n-p\rangle$, up to a fixed finite convolution.  Thus, for
$r=0,1$,
\begin{equation}\label{eq:product-charge}
 \norm{\square_{-n}(q(u_\ell^{(r)})_2)}_{L^2_{\uloc}}
 \le C(1+Q)\beta_{\ell,1}
 \sum_p\frac{\alpha_p}{\langle2\ell-n-p\rangle}.
\end{equation}
The frequency calculation behind the index relation is
\[
 \underbrace{p}_{\square_pq}
 +\underbrace{j}_{\square_j(u_\ell^{(r)})_2}
 =\underbrace{-n}_{\text{output center}}+O(1),
\]
so $j=-n-p+O(1)$ and
$j+2\ell=2\ell-n-p+O(1)$.  The $O(1)$ constants are only the fixed widths
of the three cutoff supports and are absorbed by the Japanese brackets.
Since $\bar q=\sum_\ell(u_\ell)_1$ and the cutoff is real and even,
$\square_{-n}\bar q=\overline{\square_nq}$.  Summing
\eqref{eq:first-component-two-inversions} gives
\begin{equation}\label{eq:almost-coercive}
 \alpha_n\le C\sum_\ell
 \frac{\beta_{\ell,1}}{\langle\ell-n\rangle^2}
 +C(1+Q)\sum_{\ell,p}
 \frac{\beta_{\ell,1}\alpha_p}
 {\langle\ell-n\rangle\langle2\ell-n-p\rangle}.
\end{equation}

To sum the second kernel, put
\begin{equation}\label{eq:harmonic-kernel}
 \mathcal K(d)=\sum_{r\in\Z}
 \frac1{\langle r\rangle\langle r-d\rangle}.
\end{equation}
Splitting into $|r|\le|d|/2$, $|r-d|\le|d|/2$, and the complement gives
\begin{equation}\label{eq:harmonic-bound}
 \mathcal K(d)\le C\frac{\log(2+|d|)}{\langle d\rangle},
 \qquad \mathcal K\in\ell^2(\Z).
\end{equation}
Let $\boldsymbol\beta=(\beta_{\ell,1})_{\ell\in\Z}$ and put
$\mathcal Q=\mathcal K*\boldsymbol\beta$.  Summation of
\eqref{eq:almost-coercive} yields
\begin{equation}\label{eq:good-bad-start}
 N\le CL+C(1+Q)\sum_p\alpha_p\mathcal Q_p,
 \qquad \norm{\mathcal Q}_{\ell^2}\le CL.
\end{equation}
Indeed, for fixed $\ell,p$ set $r=\ell-n$.  Then
$2\ell-n-p=r+\ell-p$, and hence
\[
 \sum_n\frac1{\langle\ell-n\rangle
                   \langle2\ell-n-p\rangle}
 =\sum_r\frac1{\langle r\rangle\langle r+\ell-p\rangle}
 =\mathcal K(\ell-p).
\]
Therefore the full triple sum is
$\sum_p\alpha_p(\mathcal K*\boldsymbol\beta)_p$; this is where the factor
two in the original denominator disappears after the change of variables.

Choose
\[
 \eta=\frac1{2C(1+Q)},\qquad
 \cG=\{p:\mathcal Q_p\le\eta\},\qquad
 \mathcal B=\Z\setminus\cG.
\]
The contribution of $\cG$ is absorbed into the left side of
\eqref{eq:good-bad-start}, since
$C(1+Q)\sum_{p\in\cG}\alpha_p\mathcal Q_p\le N/2$.
Chebyshev gives
\begin{equation}\label{eq:bad-count}
 |\mathcal B|\le C(1+Q)^2L^2.
\end{equation}
By \eqref{eq:alpha-cap}, Cauchy--Schwarz, and
\eqref{eq:good-bad-start},
\[
 \sum_{p\in\mathcal B}\alpha_p\mathcal Q_p
 \le CQ\sum_{p\in\mathcal B}\mathcal Q_p
 \le C Q|\mathcal B|^{1/2}\norm{\mathcal Q}_{\ell^2}
 \le C Q(1+Q)L^2.
\]
We have proved the exact-frame box estimate
\begin{equation}\label{eq:box-coercivity}
 N\le C\bigl[L+Q(1+Q)^2L^2\bigr].
\end{equation}

It remains to express the right side through $k=k_1$.  The proof of
Corollary~\ref{cor:amplitude} gives
\[
 L\le C(1+\sqrt Q)k,
 \qquad Q\le C L.
\]
Young's inequality applied to
$Q\le C(1+\sqrt Q)k$ gives
\begin{equation}\label{eq:Q-L-by-k}
 Q\le Ck(1+k),\qquad L\le Ck(1+k).
\end{equation}
Insert these bounds in \eqref{eq:box-coercivity}.  Since
$1+Q\le C(1+k)^2$,
\[
 N\le C\bigl[k(1+k)+k^3(1+k)^7\bigr]
 \le Ck(1+k)^9.
\]
The last exponent is deliberately rounded up for a simple uniform
polynomial.  No later step uses the number $9$ itself; the argument only
needs some fixed polynomial dependence on $k$.
Use \eqref{eq:modulation-uloc}.  Finally $k_1\le k_R$, which proves
\eqref{eq:saturated-coercivity}.  If $k_R=0$, all positive features vanish;
the first component of the dual frame gives $q=0$.
\end{proof}

\section{Large-window spectral estimates}\label{sec:large-window-spectral}

The raw current contains unsaturated positive and companion features, whereas
$k_R$ records only their saturated sum.  Spreading prevents isolated feature
spikes, and a companion convolution reduces all remaining current terms to
summable expressions in $k_R$. These estimates are uniform in the window
radius; all constants below are independent of $q$ and $R\ge1$.

The first lemma shows that one very large $\beta_n$ cannot remain invisible
to the saturated sum: nearby frame centers must also carry mass.

\begin{lemma}[Spreading of adapted features]\label{lem:spreading}
For all $m,n\in\Z$,
\begin{equation}\label{eq:reverse-spreading}
 \beta_{n,R}\le C_{\kappa,\chi}
 \langle m-n\rangle^3\beta_{m,R}.
\end{equation}
If $k=k_R(q)$, then
\begin{equation}\label{eq:beta-spreading-consequences}
 \sup_n\beta_{n,R}\le C(1+k)^3,
 \qquad
 \cL_R=\sum_n\beta_{n,R}\le Ck(1+k)^3.
\end{equation}
\end{lemma}

\begin{proof}
For the forward feature,
\[
 F_n^+=R_{nm}(A_q^+)F_m^+,
 \qquad
 R_{nm}(s)=\frac{h_\kappa(s+n)}{h_\kappa(s+m)}.
\]
If $r=m-n$, then
\begin{equation}\label{eq:ratio-symbol}
 R_{nm}(s)=
 \prod_{a\in\{2\kappa,4\kappa,6\kappa\}}
 \left(1+\frac{ir}{a+i(s+n)}\right).
\end{equation}
Put $\mathcal A=\{2\kappa,4\kappa,6\kappa\}$.  Since the three
poles are distinct, partial fractions give the explicit formula
\[
 R_{nm}(s)=1+\sum_{a\in\mathcal A}\frac{C_a(r)}{a+i(s+n)},
 \qquad
 C_a(r)=ir\prod_{\substack{b\in\mathcal A\\b\ne a}}
                  \left(1+\frac{ir}{b-a}\right).
\]
Consequently, in the convention of Lemma~\ref{lem:ulocfc},
\[
 R_{nm}(s)=\int_\R e^{-its}\dd\mu_{nm}(t),\qquad
 \dd\mu_{nm}(t)=\dd\delta_0(t)
  +\one_{(0,\infty)}(t)e^{-int}
       \sum_{a\in\mathcal A}C_a(r)e^{-at}\dd t.
\]
Equivalently, the Fourier transform with our standard convention is
\[
 \widehat R_{nm}(t)=2\pi\delta_0(t)
  +2\pi\one_{(-\infty,0)}(t)e^{int}
       \sum_{a\in\mathcal A}C_a(r)e^{at}.
\]
Here the coefficients are polynomials in $r$ of degree at most three;
no polynomial factors in $t$ are needed because the poles are simple.
Since $|C_a(r)|\le C_\kappa\langle r\rangle^3$, we obtain
\begin{equation}\label{eq:ratio-moment}
 \int_\R(1+|t|)\dd|\mu_{nm}|(t)
 \le 1+\sum_{a\in\mathcal A}|C_a(r)|
                  \left(\frac1a+\frac1{a^2}\right)
 \le C_\kappa\langle r\rangle^3.
\end{equation}
Apply Lemma~\ref{lem:ulocfc}.  This proves \eqref{eq:reverse-spreading}.

 Let $M_0=\sup_n\beta_{n,R}$.  It is finite by
Proposition~\ref{prop:feature-summability}.  If $M_0>1$, choose $n_0$ with
$\beta_{n_0,R}\ge M_0/2$.  Reversing \eqref{eq:reverse-spreading} gives
\[
 \beta_{m,R}\ge cM_0\langle m-n_0\rangle^{-3}.
\]
Thus $\beta_{m,R}\ge1$ whenever
$\langle m-n_0\rangle\le c_\kappa M_0^{1/3}$, which supplies at least
$c_\kappa M_0^{1/3}$ integer indices (after harmless adjustment when this
number is bounded).  Equation
\eqref{eq:k-beta} gives $M_0\le C(1+k)^3$.  The features below one
sum to $O(k)$.  There are $O(k)$ remaining indices and each is at most
$M_0$.  This proves \eqref{eq:beta-spreading-consequences}.
\end{proof}

The current also contains the companion norm $\gamma_{n,R}$.  The next lemma
expresses it as a summable convolution of the positive features, so no new
control quantity is needed.

\begin{lemma}[Companion convolution]\label{lem:companion-convolution}
For every $n$,
\begin{equation}\label{eq:companion-convolution}
 \gamma_{n,R}\le C_{\kappa,\chi}
 \sum_{\ell\in\Z}\langle n-\ell\rangle^{-2}\beta_{\ell,R}.
\end{equation}
\end{lemma}

\begin{proof}
Insert the dual frame into $G_n^+=t_\kappa(A_q^++n)\bar q e_1$.
The multiplier carrying $F_\ell^+$ to the $n$th companion is
\begin{equation}\label{eq:companion-symbol}
 a_{n\ell}(s)=t_\kappa(s+n)r_\ell(s).
\end{equation}
By \eqref{eq:companion-wiener},
$\|a_{n\ell}\|_{\mathcal W_1}\le C_\kappa\langle n-\ell\rangle^{-2}$.
Lemma~\ref{lem:ulocfc} and absolute feature summability therefore give
\[
 \|G_n^+\|_{U_R}
 \le C_\kappa\sum_\ell\langle n-\ell\rangle^{-2}\|F_\ell^+\|_{U_R}.
\]
The constant is uniform in $R\ge1$ because
$(1+|t|/R)^{1/2}\le(1+|t|)^{1/2}$ in the uniformly local calculus.
\end{proof}

The saturation compensates for the loss of one detuning power in the
companion feature.

\begin{lemma}[Saturated sequence sums]\label{lem:saturated-sums}
Let $k=k_R(q)$.  Then
\begin{align}
 I_{b,R}&:=\sum_n\frac{\beta_{n,R}}
 {(1+\beta_{n,R}^2)^{3/2}}
 \le Ck,\label{eq:Ib}\\
 I_{c,R}&:=\sum_n\frac{\gamma_{n,R}}
 {(1+\beta_{n,R}^2)^{3/2}}
 \le C(k+k^3).\label{eq:Ic}
\end{align}
\end{lemma}

\begin{proof}
The first estimate follows from
$b(1+b^2)^{-3/2}\le\min(b,1)$ and \eqref{eq:k-beta}.
For the second one, insert \eqref{eq:companion-convolution}, interchange the
nonnegative sums, and fix $\ell$.  With
$M=\beta_{\ell,R}$ and $d=\langle n-\ell\rangle$, reverse spreading gives
$\beta_{n,R}\ge cd^{-3}M$.  Hence
\begin{align}
 \sum_n\frac{\langle n-\ell\rangle^{-2}}
 {(1+\beta_{n,R}^2)^{3/2}}
 &\le C\sum_{d\ge1}d^{-2}\min(1,d^9M^{-3})\notag\\
 &\le C\begin{cases}1,&M\le1,\\M^{-1/3},&M>1.\end{cases}
 \label{eq:companion-inner-sum}
\end{align}
The split is at $d\simeq M^{1/3}$.  Multiplication by the outer factor
$M$ gives
\[
 I_{c,R}\le C\sum_\ell\min(\beta_{\ell,R},
                              \beta_{\ell,R}^{2/3}).
\]
The small terms sum to $O(k)$.  There are $O(k)$ large terms, and
indeed each index with $\beta_{\ell,R}>1$ contributes a fixed positive
constant to the saturated sum \eqref{eq:exact-summand}.  Hence their
cardinality is at most $Ck$.  Lemma~\ref{lem:spreading} bounds each
$2/3$ power by $C(1+k)^2$.
This proves \eqref{eq:Ic}.
\end{proof}

\section{Active shear and the large-window estimate}\label{sec:active-flux}

This section differentiates $K_R$ despite its nonattained spatial suprema and
infinite frequency sum.  The raw derivative-free current estimate and the
spectral estimates of the previous section yield an $R^{-1}$ differential
inequality, which is then integrated for the endpoint bootstrap.

The spatial supremum in \eqref{eq:KR} need not be attained.  We use the
following form of Danskin's theorem.

\begin{lemma}[Noncompact Danskin theorem]\label{lem:danskin}
Let $B:[t_0-\delta,t_0+\delta]\times\R\to\R$ be bounded and continuous,
continuously differentiable in $t$, and suppose
$\sup_y|\partial_tB(t_0,y)|<\infty$ and
\begin{equation}\label{eq:uniform-time-derivative}
 \sup_y|\partial_tB(t,y)-\partial_tB(t_0,y)|\longrightarrow0
 \quad(t\to t_0).
\end{equation}
Set $M(t)=\sup_yB(t,y)$ and
\[
 \cM^\varepsilon(t_0)
 =\{y:B(t_0,y)\ge M(t_0)-\varepsilon\}.
\]
Then
\begin{equation}\label{eq:noncompact-danskin}
 D_t^+M(t_0)
 \le\limsup_{\varepsilon\downarrow0}
 \sup_{y\in\cM^\varepsilon(t_0)}\partial_tB(t_0,y).
\end{equation}
\end{lemma}

\begin{proof}
Assumption \eqref{eq:uniform-time-derivative} supplies the neighborhood
bound used below: for $|t-t_0|$ sufficiently small,
\[
 \sup_y|\partial_tB(t,y)|
 \le \sup_y|\partial_tB(t_0,y)|+1.
\]
For $h>0$ choose $y_h$ such that
$B(t_0+h,y_h)\ge M(t_0+h)-h^2$.  Uniform boundedness of $\partial_tB$
gives $|M(t_0+h)-M(t_0)|\le Ch$, and therefore
$y_h\in\cM^{\varepsilon_h}(t_0)$ for some $\varepsilon_h\to0$.  Moreover,
\begin{align*}
 \frac{M(t_0+h)-M(t_0)}h
 &\le\frac{B(t_0+h,y_h)-B(t_0,y_h)}h+h\\
 &=\frac1h\int_{t_0}^{t_0+h}\partial_tB(s,y_h)\dd s+h.
\end{align*}
Use \eqref{eq:uniform-time-derivative} and take the upper limit.
The points $y_h$ may escape to infinity.  No subsequence or compactness in
$y$ is used: both the neighborhood bound and the convergence in
\eqref{eq:uniform-time-derivative} are uniform over the entire real line.
\end{proof}

Since $k_R$ is a countable sum of nonlinear functions of nonattained
suprema, we must also justify composition with the saturation and passage to
the infinite sum.

\begin{lemma}[Composition and countable summation]\label{lem:danskin-sum}
Let $J$ be a compact interval with nonempty interior, and let $B_n\ge0$
satisfy the hypotheses of Lemma~\ref{lem:danskin} at every time in a
neighborhood of $J$.  Let
$\Phi\in C^1((0,\infty))\cap C([0,\infty))$ be increasing with
$\Phi(0)=0$.  Suppose there are nonnegative, measurable, locally integrable
functions $d_n(t)$ on that neighborhood such that, for $t\in J$, uniformly in
$0<\eta\le1$,
\begin{equation}\label{eq:danskin-majorant}
 \sup_y\Phi'(B_n(t,y)+\eta)|\partial_tB_n(t,y)|\le d_n(t),
\end{equation}
and
\begin{equation}\label{eq:danskin-uniform-tail}
 \lim_{N\to\infty}\sup_{t\in J}\sum_{|n|>N}d_n(t)=0.
\end{equation}
Put
\[
 M_n(t)=\sup_yB_n(t,y),\qquad
 \cM_n^\varepsilon(t)=
 \{y:B_n(t,y)\ge M_n(t)-\varepsilon\}.
\]
Assume that, for each $n$, either $M_n(t)>0$ for every $t\in J$ or
$M_n\equiv0$ on $J$.  Assume also that
$\sum_n\Phi(M_n(t_*))<\infty$ at one time $t_*\in J$.
For a positive box define
\[
 a_n(t)=\limsup_{\varepsilon\downarrow0}
 \sup_{y\in\cM_n^\varepsilon(t)}
 \Phi'(B_n(t,y))\partial_tB_n(t,y),
\]
where only $0<\varepsilon<M_n(t)/2$ is needed, so every argument of
$\Phi'$ is positive.  For an identically zero box set
\begin{equation}\label{eq:active-zero-definition}
 a_n(t)=0\qquad\text{if }M_n\equiv0\text{ on }J.
\end{equation}
Then $F(t)=\sum_n\Phi(M_n(t))$ is finite and continuous on $J$, and for
every $t\in\operatorname{int}J$,
\begin{equation}
 D_t^+F(t)\le\sum_na_n(t),
 \label{eq:countable-danskin}
\end{equation}
where the series on the right converges absolutely.
\end{lemma}

\begin{proof}
Regularization will be used only to estimate increments.  Put
$\Phi_\eta(B)=\Phi(B+\eta)-\Phi(\eta)$.  For fixed $n,y$ and $s<t$ in
$J$, the fundamental theorem of calculus and
\eqref{eq:danskin-majorant} give
\[
 |\Phi_\eta(B_n(t,y))-\Phi_\eta(B_n(s,y))|
 \le\int_s^td_n(r)\dd r.
\]
Since $\Phi_\eta$ is increasing and continuous,
$\sup_y\Phi_\eta(B_n(t,y))=\Phi_\eta(M_n(t))$.  Taking suprema in the
two directions and then letting $\eta\downarrow0$ for these fixed times
gives, with $f_n(t)=\Phi(M_n(t))$,
\begin{equation}\label{eq:danskin-increment-bound}
 |f_n(t)-f_n(s)|\le\int_s^td_n(r)\dd r.
\end{equation}
This step does not pass a Dini derivative through the regularization limit.

Write
\[
 \epsilon_N=\sup_{t\in J}\sum_{|n|>N}d_n(t)\longrightarrow0.
\]
For sufficiently large $N$, the tail is bounded by the finite constant
$\epsilon_N$, while each of the finitely many remaining $d_n$ is
integrable on $J$.  Thus $\sum_nd_n$ is integrable on $J$.
Summing \eqref{eq:danskin-increment-bound} from $t_*$ proves that $F$ is
finite everywhere on $J$.  Moreover, if
$E_N(t)=\sum_{|n|>N}f_n(t)$, then
\begin{equation}\label{eq:danskin-tail-increment}
 |E_N(t)-E_N(s)|\le |t-s|\epsilon_N.
\end{equation}
The tails at $t_*$ tend to zero, so this bound also proves uniform
convergence of the series defining $F$, and hence continuity.

We now differentiate each fixed summand without regularization.  The
uniform time-derivative hypothesis in Lemma~\ref{lem:danskin} makes $M_n$
locally Lipschitz.  If $M_n(t)>0$, differentiability of $\Phi$ at $M_n(t)$
and this local Lipschitz bound give
\[
 f_n(t+h)-f_n(t)
 =\Phi'(M_n(t))(M_n(t+h)-M_n(t))+o(|h|).
\]
On $\cM_n^\varepsilon(t)$, $B_n(t,y)$ tends uniformly to $M_n(t)$ as
$\varepsilon\downarrow0$.  Continuity of $\Phi'$ near that positive
value and boundedness of $\partial_tB_n(t,\cdot)$ therefore allow
$\Phi'(M_n(t))$ to be replaced by $\Phi'(B_n(t,y))$ in the active limit.
Lemma~\ref{lem:danskin} yields $D_t^+f_n(t)\le a_n(t)$.
If $M_n\equiv0$, then $f_n\equiv0$ and the same inequality holds with
$a_n=0$.  For a positive box, letting $\eta\downarrow0$ in
\eqref{eq:danskin-majorant} at positive arguments shows
$|a_n(t)|\le d_n(t)$; this bound also holds for a zero box.
Consequently $\sum_na_n(t)$ is absolutely convergent.

For each finite cutoff, subadditivity of the upper right Dini derivative
and \eqref{eq:danskin-tail-increment} give
\[
 D_t^+F(t)
 \le D_t^+\sum_{|n|\le N}f_n(t)+\epsilon_N
 \le\sum_{|n|\le N}a_n(t)+\epsilon_N.
\]
Letting $N\to\infty$ proves \eqref{eq:countable-danskin}.  No lower
bound on $M_n(t)$ uniform in $n$ is required: differentiation is performed
at each fixed positive box, and the uniform increment bound controls the
remaining tail.
\end{proof}

The zero-persistence hypothesis in this lemma is essential.  For example,
on $[-1,1]$ take $B_0(t,y)=t^2$, all other $B_n=0$, and
$\Phi(s)=\sqrt{s/(1+s)}$.  The regularized derivative is bounded by
$d_0=1$, with all other $d_n=0$, because
\[
 \Phi'(t^2+\eta)|2t|
 =\frac{|t|}{\sqrt{t^2+\eta}(1+t^2+\eta)^{3/2}}\le1.
\]
Nevertheless $D_t^+\Phi(M_0(t))|_{t=0}=1$, whereas
$\Phi'(B_0(0,y)+\eta)\partial_tB_0(0,y)=0$ for every $\eta>0$.
Thus a regularized value at an isolated zero cannot define the required
Dini bound.  The determinant avoids this obstruction by the following
rigidity property.

\begin{lemma}[Rigidity of a zero feature box]\label{lem:zero-box-rigidity}
For $p\in M$, $n\in\Z$, $R\ge1$, and $y\in\R$,
\[
 B_{n,R}[p](y)=0\quad\Longleftrightarrow\quad p=0.
\]
Consequently, along a mild solution on a connected lifespan $I$, either
$q\equiv0$ and every box vanishes identically, or
$B_{n,R}(t,y)>0$ for every $t\in I$, $n\in\Z$, and $y\in\R$.
\end{lemma}

\begin{proof}
If one windowed integral is zero, positivity of $\chi_{R,y}$ everywhere
and of every summand in \eqref{eq:positive-density} imply
$S_m[M_np]=0$ almost everywhere for every $m\ge1$.
Only $S_1=S_2=0$ is needed.  Put $v=M_np$ and
$z_a=g_{a\kappa,+}[v]$ for $a=1,2,3$.  The normalized Green functions
are locally Lipschitz, so their almost-everywhere identities extend
everywhere.  Thus
\[
 z_1-2z_2+z_3=0,\qquad z_1^2-2z_2^2+z_3^2=0.
\]
Substitution of $z_3=2z_2-z_1$ into the second identity gives
$2(z_1-z_2)^2=0$, hence $z_1=z_2=z_3=:z$.  Subtracting
\eqref{eq:riccati-plus} at heights $\kappa$ and $2\kappa$ gives
$0=2\kappa z$ almost everywhere.  Therefore $z=0$, and its Riccati
equation yields $\bar v=0$ almost everywhere.  Since modulation is
invertible, $p=0$.
Conversely, the zero potential has zero Green ratios and zero density.

For a mild solution the set $\{t\in I:q(t)=0\}$ is relatively closed
by continuity in $M$.  It is also relatively open: local uniqueness from
Proposition~\ref{prop:local}, applied both forwards and backwards from a
zero datum, identifies the solution with zero in a neighborhood of that
time.  Connectedness of $I$ makes this set either empty or all of $I$.
Apply the first assertion to $p=q(t)$ and use the static--dynamic window
identity \eqref{eq:static-dynamic-window}.  This argument uses only local
uniqueness and the Riccati equations, not the large-window differential
inequality or global existence.
\end{proof}

For the determinant, set
\begin{equation}\label{eq:active-set}
 \cM_{n,R}^\varepsilon(t)
 =\{y:B_{n,R}(t,y)\ge\sup_zB_{n,R}(t,z)-\varepsilon\},
\end{equation}
and
\begin{equation}\label{eq:raw-flux}
 \cF_{n,R}(t,y)
 =\int(\chi_{R,y}^2)'J_\kappa[q_n](t,x)\dd x.
\end{equation}
The positive signed active shear and its absolute-flux majorant are
\begin{align}
 \cS_{\kappa,R}(t)
 &=\frac1\kappa\sum_n\limsup_{\varepsilon\downarrow0}
 \sup_{y\in\cM_{n,R}^\varepsilon(t)}
 \left[\widehat\Theta_\kappa'(B_{n,R}(t,y))
       \cF_{n,R}(t,y)\right]_+,
 \label{eq:signed-shear}\\
 \cC_{\kappa,R}(t)
 &=\frac1\kappa\sum_n\limsup_{\varepsilon\downarrow0}
 \sup_{y\in\cM_{n,R}^\varepsilon(t)}
 \widehat\Theta_\kappa'(B_{n,R}(t,y))
       |\cF_{n,R}(t,y)|.
 \label{eq:absolute-shear}
\end{align}
For the identically zero solution define both shears to be zero.
Lemma~\ref{lem:zero-box-rigidity} shows that this is the only solution
with a zero box.  Every box of a nonzero solution is strictly positive,
so the displayed derivatives of the saturation are evaluated only at
positive arguments.  For the zero solution the current and the raw flux
vanish identically as well.  Thus its zero shear agrees with any
regularization, but no derivative of a limiting regularization is used.
For mild solutions all functions of $(t,y)$ in these formulas are
continuous.  Each active supremum is Borel measurable: write the active
limit as the infimum over $k\ge1$ of the supremum over
$\{y:B_{n,R}(t,y)>\sup_zB_{n,R}(t,z)-1/k\}$ and restrict both spatial
suprema to $y,z\in\mathbb Q$.  Hence $\cS_{\kappa,R}$ and
$\cC_{\kappa,R}$, being nonnegative countable sums, are measurable and
their time integrals on compact intervals are well-defined once the
majorant below is established.
To justify the rational restriction, continuity gives
$\sup_{z\in\R}B(t,z)=\sup_{z\in\mathbb Q}B(t,z)$.  For fixed $k$, the strict
active set $\{y:B(t,y)>\sup_zB(t,z)-1/k\}$ is open; continuity of the flux
expression therefore makes its supremum equal to the supremum over rational
$y$ in that open set.  Rational evaluations are Borel functions of $t$,
countable suprema are Borel, and the final active limsup is a countable
infimum over $k$.  This proves the claimed measurability.

We now combine the active-supremum calculus with the saturated sequence
bounds to obtain the differential estimate driving the global bootstrap.

\begin{proposition}[Large-window active flux]\label{prop:large-window-flux}
Along every mild $M$ solution,
\begin{equation}\label{eq:Dini-chain}
 D_t^+k_R(t)\le\cS_{\kappa,R}(t)\le\cC_{\kappa,R}(t)
\end{equation}
and
\begin{equation}\label{eq:large-window-differential}
 \cC_{\kappa,R}(t)
 \le\frac{C_{\kappa,\chi}}R
 k_R(t)(1+k_R(t))^2.
\end{equation}
\end{proposition}

\begin{proof}
If $q$ vanishes at any time, Lemma~\ref{lem:zero-box-rigidity} gives the
identically zero solution, for which all the assertions are immediate.
Otherwise every $M_n(t)=\sup_yB_{n,R}(t,y)$ is strictly positive on the
lifespan.  Fix a compact interval $J$ inside that lifespan, with the time
under consideration in its interior.  For Lemma~\ref{lem:danskin-sum} take
$B_n=B_{n,R}$ and
$\Phi=\kappa^{-1}\widehat\Theta_\kappa$.  For each fixed $n$, $q_n\in C(J;M)$: translations are strongly
continuous on $M$ by the same finite-box approximation used for the free
group in Lemma~\ref{lem:modulation}.  Proposition~\ref{prop:micro-law}
therefore gives continuous maps $t\mapsto E_\kappa[q_n(t)]$ and
$t\mapsto J_\kappa[q_n(t)]$ into $L^\infty$.
Test its distributional law against the window, first with an additional
compact spatial cutoff.  Boundedness of the density and current on $J$
and integrability of the window and its derivative permit removal of this
cutoff.  Consequently, for every $s,t\in J$,
\[
 B_{n,R}(t,y)-B_{n,R}(s,y)
 =\int_s^t\int(\chi_{R,y}^2)'J_\kappa[q_n(r)]\dd x\dd r.
\]
The identity first holds for each $y$ and then for all $y$ by continuity.
Its right side is continuously differentiable in $t$, uniformly in $y$,
since
\[
 \sup_y|\partial_tB_{n,R}(t,y)-\partial_tB_{n,R}(s,y)|
 \le\|(\chi_{R,0}^2)'\|_1
       \|J_\kappa[q_n(t)]-J_\kappa[q_n(s)]\|_\infty\longrightarrow0.
\]
Thus \eqref{eq:uniform-time-derivative} holds even though the spatial
supremum need not be attained.  The positivity hypothesis of
Lemma~\ref{lem:danskin-sum} has already been verified by
Lemma~\ref{lem:zero-box-rigidity}.  The initial finiteness
required in Lemma~\ref{lem:danskin-sum} follows from
Proposition~\ref{prop:feature-summability}.  The passage to the infinite
sum follows from the following majorant.  Since
$\widehat\Theta_\kappa'$ is decreasing,
Lemma~\ref{lem:raw-flux-bound} gives, uniformly in $y$ and
$0<\eta\le1$,
\begin{equation}\label{eq:flux-tail-majorant}
 \frac1\kappa\widehat\Theta_\kappa'(B_{n,R}+\eta)
       |\cF_{n,R}|
 \le\frac C R\bigl(\gamma_{n,R}
              +\norm q_\infty\beta_{n,R}\bigr)=:d_n(t).
\end{equation}
Proposition~\ref{prop:feature-continuity} and
$M\hookrightarrow L^\infty$ show that $t\mapsto(d_n(t))_n$ is continuous
into $\ell^1$.  On compact time intervals its tails are therefore uniform,
which verifies \eqref{eq:danskin-uniform-tail}.
For the only product term, writing $Q(t)=\norm{q(t)}_\infty$ gives the exact
estimate
\[
 \|Q(t)\boldsymbol\beta(t)-Q(s)\boldsymbol\beta(s)\|_{\ell^1}
 \le |Q(t)-Q(s)|\|\boldsymbol\beta(t)\|_{\ell^1}
     +Q(s)\|\boldsymbol\beta(t)-\boldsymbol\beta(s)\|_{\ell^1}.
\]
The embedding $M\hookrightarrow L^\infty$ makes $Q$ continuous, and
Proposition~\ref{prop:feature-continuity} handles the other two factors.
All hypotheses of Lemma~\ref{lem:danskin-sum} now hold.  It gives
$D_t^+k_R\le\sum_na_n$, where $a_n$ is the signed active limit for
$\Phi=\kappa^{-1}\widehat\Theta_\kappa$.  Replacing each active
expression by its positive part, and then by its absolute value, gives
\eqref{eq:Dini-chain}.  The same majorant $d_n$ controls all these
series.  In particular, the regularization in
\eqref{eq:flux-tail-majorant} is used only for the increment and tail
bounds; the active Dini derivatives are taken at fixed positive boxes.

For $B=30\kappa b^2$, direct differentiation gives the exact cancellation
\begin{equation}\label{eq:saturation-cancellation}
 \widehat\Theta_\kappa'(B)\,\kappa b
 =\frac\kappa{2\sqrt{30}}(1+30b^2)^{-3/2}.
\end{equation}
Indeed,
\[
 \widehat\Theta_\kappa'(B)
 =\frac{\kappa^2}{2\sqrt B(\kappa+B)^{3/2}},
\]
and substitution of
$\sqrt B=\sqrt{30\kappa}\,b$ and
$\kappa+B=\kappa(1+30b^2)$ gives
\eqref{eq:saturation-cancellation} after multiplication by $\kappa b$.
At an asymptotically maximizing window, $b_{n,R}(y)\to\beta_{n,R}$ and
$c_{n,R}(y)\le\gamma_{n,R}$.
Lemma~\ref{lem:raw-flux-bound} and
\eqref{eq:saturation-cancellation} yield
\begin{equation}\label{eq:flux-Ib-Ic}
 \cC_{\kappa,R}
 \le\frac CR\left(I_{c,R}+QI_{b,R}\right),
 \qquad Q=\norm q_\infty.
\end{equation}
Lemmas~\ref{lem:saturated-sums} and Corollary~\ref{cor:amplitude} give
\[
 I_{c,R}+QI_{b,R}
 \le C(k_R+k_R^3)+C(k_R+k_R^2)k_R
 \le Ck_R(1+k_R)^2.
\]
This proves \eqref{eq:large-window-differential}.
\end{proof}

To integrate the upper-Dini differential inequality for \(k_R\), we use the
following elementary Gronwall-type comparison.

\begin{lemma}[Dini comparison]\label{lem:Dini-comparison}
Let $f:(a,b)\to(0,\infty)$ be continuous.  If
$[s,s+\sigma]\subset(a,b)$ and
$D_t^+\log f(t)\le A$ on $[s,s+\sigma]$, then
\begin{equation}\label{eq:Dini-Gronwall}
 f(t)\le f(s)e^{A(t-s)},\qquad s\le t\le s+\sigma.
\end{equation}
\end{lemma}

\begin{proof}
Put $h(t)=\log f(t)-A(t-s)$, so $D_t^+h(t)\le0$.  If
$h(t_1)>h(s)$ for some $t_1>s$, choose
$0<\varepsilon<(h(t_1)-h(s))/(t_1-s)$ and set
$g(t)=h(t)-\varepsilon(t-s)$.  Then $g(t_1)>g(s)$, so a point at which
$g$ attains its minimum on $[s,t_1]$ occurs at some $t_0<t_1$: the right
endpoint cannot be a minimizer because $g(t_1)>g(s)$.  The point $t_0$ may
equal the left endpoint $s$, but that causes no problem because the
derivative is an upper \emph{right} Dini derivative.  For every sufficiently
small $r>0$, $g(t_0+r)\ge g(t_0)$, and hence
$D_t^+g(t_0)\ge0$.  On the other hand,
$D_t^+g(t_0)\le-\varepsilon$, a contradiction.  Hence $h(t)\le h(s)$,
which is \eqref{eq:Dini-Gronwall}.
\end{proof}

Applying the comparison lemma to the large-window active-flux estimate gives
the collision-duration bound used in the proof of
Theorem~\ref{thm:main}.

\begin{proposition}[Large-window bootstrap]\label{prop:integrated-shear}
Let $q$ be a mild $M$ solution, let $[s,s+\sigma]$ be contained in its
lifespan, and put $k_s=k_R(q(s))$.
If
\begin{equation}\label{eq:collision-condition}
 \frac{C\sigma}{R}(1+2k_s)^2<\log2,
\end{equation}
then
\begin{equation}\label{eq:integrated-bootstrap}
 \sup_{s\le t\le s+\sigma}k_R(q(t))\le2k_s.
\end{equation}
\end{proposition}

\begin{proof}
If $q(s)=0$, uniqueness gives the zero solution.  Otherwise
Lemma~\ref{lem:zero-box-rigidity} gives $k_R(q(t))>0$ throughout the
lifespan.  On the first interval on which
$k_R\le2k_s$, Proposition~\ref{prop:large-window-flux} gives
\[
 D_t^+\log k_R(t)\le\frac CR(1+2k_s)^2.
\]
Lemma~\ref{lem:Dini-comparison} and \eqref{eq:collision-condition} exclude
a first exit through $2k_s$.  This proves \eqref{eq:integrated-bootstrap}.
\end{proof}

\section{The perturbative frequency envelope}\label{sec:frequency-envelope}

To start the large-window bootstrap, we need an initial bound for $k_R$ that
is much smaller than the full coercive bound.  A small NLS dilation collapses
the original frequency centers, while a compensating window preserves the
relevant spatial mass.  It is important that the estimate below retains the
centers of the original frequency boxes.  We write
\[
 w(s)=\langle s\rangle^{-3},\qquad s\in\R.
\]

The next lemma separates each nonlinear adapted feature into its centered
linear contribution and an $\ell^1$-summable quadratic remainder.

\begin{lemma}[Centered feature estimate]
\label{lem:centered-feature}
There is \(\delta_\kappa>0\) with the following property.  Suppose that
\(p\in M\) and \(\norm p_M\le\delta_\kappa\), and put
\[
 a_j(p)=\norm{\square_jp}_{L^\infty}.
\]
Let $\beta_{n,1}(p)$ be the static feature norm
\[
 \beta_{n,1}(p)=\norm{F_n^+[p]}_{U_1},
\]
equivalently the quantity in \eqref{eq:b-feature} at time $t=0$; no
solution flow is involved.
Then there is a nonnegative sequence \(\rho=(\rho_n)\) such that
\begin{align}
 \beta_{n,1}(p)
 &\le C_\kappa\sum_{j\in\Z}w(n-j)a_j(p)+\rho_n,
 \label{eq:centered-envelope}\\
 \sum_{n\in\Z}\rho_n
 &\le C_\kappa\norm p_M^2.
 \label{eq:centered-remainder}
\end{align}
The same assertions hold if the fixed unit window is replaced by any fixed
slightly larger Schwartz window.
\end{lemma}

\begin{proof}
Put $N=\|p\|_M$, $a_j=a_j(p)$, and use the centered Riccati envelope
from the proof of Proposition~\ref{prop:feature-summability}:
\[
 A_n=\sum_{j\in\Z}\frac{a_j}{\kappa+|j-n|}.
\]
Since $A_n\le N/\kappa$, choosing $\delta_\kappa\le1$ with
$\delta_\kappa^2/\kappa\le c_\kappa$ ensures
$NA_n\le c_\kappa$ for every $n$.  Hence every index satisfies the good
condition used to prove \eqref{eq:riccati-error-bounds},
\eqref{eq:first-grade-feature}, and \eqref{eq:higher-grade-feature}.
Those estimates were obtained directly from the Riccati contraction and
Schur bounds and do not depend on the present lemma.

In particular, retaining the cubic decay of the positive first-grade
features in \eqref{eq:first-grade-feature}, rather than the weaker
quadratic decay used there to include the companions, gives
\[
 \|S_1[M_np]\|_{U_1}
 \le C_\kappa\sum_j\langle n-j\rangle^{-3}a_j
      +C_\kappa N A_n^2.
\]
The positive part of \eqref{eq:higher-grade-feature} gives
\[
 \left(\sum_{m\ge2}
  \|S_m[M_np]\|_{U_1}^2\right)^{1/2}
 \le C_\kappa A_n^2.
\]
For each window, split off the first grade and bound the higher-grade
norm by the square sum of the individual $U_1$ norms.  Taking the
spatial supremum then yields
\begin{equation}\label{eq:centered-riccati-envelope}
 \beta_{n,1}(p)
 \le C_\kappa\sum_jw(n-j)a_j+C_\kappa(1+N)A_n^2.
\end{equation}
This estimate concerns exactly the positive features defining
$\beta_{n,1}$; the companion terms are unnecessary here.

Let $v_\kappa(n)=(\kappa+|n|)^{-1}$.  Since
$v_\kappa\in\ell^2(\Z)$, discrete Young's inequality gives
\[
 \|A\|_{\ell^2}=\|v_\kappa*a\|_{\ell^2}
 \le\|v_\kappa\|_{\ell^2}\|a\|_{\ell^1}
 \le C_\kappa N.
\]
Define $\rho_n=C_\kappa(1+N)A_n^2$, choosing the constant as in
\eqref{eq:centered-riccati-envelope}.  Then
\[
 \sum_n\rho_n
 \le C_\kappa(1+N)\|A\|_{\ell^2}^2
 \le C_\kappa N^2,
\]
where $N\le\delta_\kappa\le1$ is used in the last step.  These are
\eqref{eq:centered-envelope} and \eqref{eq:centered-remainder}, including
$p=0$.  Replacing the unit window by a fixed larger Schwartz window only
changes the constants: the first- and higher-grade estimates in
Section~\ref{sec:fock} follow from $M$ bounds and the elementary estimate
$\|\chi_y f\|_2\le\|\chi\|_2\|f\|_\infty$.
\end{proof}

Under a small dilation, the centered linear envelope converges to the single
profile $Nw$, while the nonlinear remainder vanishes.  This produces the
uniform initial $k_R$ bound needed in the proof of
Theorem~\ref{thm:main}.

\begin{proposition}[Frequency collapse in a compensating window]
\label{prop:frequency-collapse}
Let \(q\in M\), \(N=\norm q_M\), and \(H\ge1\).  For all sufficiently
small \(0<\lambda\le1\), depending on \(q,H,\kappa,\chi\),
\begin{equation}\label{eq:collapse-bound}
 k_{H/\lambda^2}(S_\lambda q)
 \le C_{\kappa,\chi}\bigl(1+H^{1/6}N^{1/3}\bigr).
\end{equation}
The same \(\lambda\) may be used for every datum in a sufficiently small
\(M\)-neighborhood of \(q\).
\end{proposition}

\begin{proof}
Put \(a_j=\norm{\square_jq}_\infty\).  Each function
\(\lambda(\square_jq)(\lambda\,\cdot)\) has frequency width \(O(\lambda)\),
center \(\lambda j\), and \(L^\infty\)-norm at most \(\lambda a_j\).
Consequently, if \(p=S_\lambda q\), then
\begin{align}
 \norm p_M&\le C\lambda N,
 \label{eq:scaled-M-small}\\
 a_\ell(p)&\le C\lambda
 \sum_j\one_{\{|\ell-\lambda j|\le C\}}a_j.
 \label{eq:scaled-boxes}
\end{align}
For \eqref{eq:scaled-boxes}, decompose
$p=\sum_j\lambda(\square_jq)(\lambda\cdot)$ before applying
$\square_\ell$ and use the triangle inequality.  Each scaled input support
has width $O(\lambda)\le O(1)$ and therefore meets only a fixed number of
output unit boxes.  An output box may receive many inputs, which is exactly
why the right side is a sum rather than a supremum.  If one later sums in
$\ell$, reversing the order counts each input $j$ only a fixed number of
times, so there is no hidden overcount.
Choose \(\lambda\) so small that Lemma~\ref{lem:centered-feature} applies.
With \(R=H/\lambda^2\), Lemma~\ref{lem:windows} gives
\begin{equation}\label{eq:scaled-feature-envelope}
 \beta_{n,R}(p)
 \le C\sqrt H\sum_jw(n-\lambda j)a_j+e_n^{(\lambda)},
 \qquad
 \sum_ne_n^{(\lambda)}\le C\sqrt H\,\lambda N^2.
\end{equation}
Indeed, the linear factor is \(\sqrt R\lambda=\sqrt H\), while the
nonlinear remainder has size
\(\sqrt R(\lambda N)^2=\sqrt H\lambda N^2\).
For the linear part, substitute \eqref{eq:scaled-boxes}, interchange the
nonnegative sums, and use that the finitely many $\ell$ satisfying
$|\ell-\lambda j|\le C$ obey
$w(n-\ell)\le Cw(n-\lambda j)$.  This gives exactly
\[
 \sqrt R\sum_\ell w(n-\ell)a_\ell(p)
 \le C\sqrt R\,\lambda\sum_jw(n-\lambda j)a_j
 =C\sqrt H\sum_jw(n-\lambda j)a_j.
\]

For every fixed \(j\), translation continuity in \(\ell^1(\Z)\) gives
\[
 \sum_n|w(n-\lambda j)-w(n)|\longrightarrow0
 \qquad(\lambda\downarrow0).
\]
To verify this elementary fact, first choose $J$ so that the two tails
$|n|>J$ are uniformly small for shifts of size at most one, using
$w(n-\theta)\lesssim\langle n\rangle^{-3}$.  On the remaining finite set,
uniform continuity of $w$ makes every summand small.
These sums are uniformly bounded.  Dominated convergence with the
summable weights \(a_j\) therefore proves
\begin{equation}\label{eq:l1-frequency-collapse}
 \sum_ja_jw(n-\lambda j)\longrightarrow
 \left(\sum_ja_j\right)w(n)
 \quad\text{in }\ell_n^1.
\end{equation}
More precisely,
\[
 \sup_{\theta\in\R}\sum_{n\in\Z}w(n-\theta)<\infty,
\]
by comparison with $\sum_n\langle n\rangle^{-3}$ after reducing
$\theta$ modulo one.  Hence
\[
 \sum_n|w(n-\lambda j)-w(n)|\le C
\]
uniformly in both $j$ and $\lambda$.  For each fixed $j$ this quantity tends
to zero by continuity of translations in sampled $\ell^1$, so
$Ca_j$ is the required summable dominating sequence.
The map \(z\mapsto\min(z,1)\) is one-Lipschitz on \([0,\infty)\), and
\begin{equation}\label{eq:occupancy-sum}
 \sum_{n\in\Z}\min(s\langle n\rangle^{-3},1)
 \le C(1+s^{1/3}),\qquad s\ge0.
\end{equation}
For $s>1$, split at $J=\lceil s^{1/3}\rceil$.  The region $|n|\le J$ has
$O(J)$ terms, while the tail is at most
$Cs\sum_{|n|>J}|n|^{-3}\le CsJ^{-2}\le Cs^{1/3}$.
For $0\le s\le1$, summing the unsaturated sequence gives a uniform bound.
Use \eqref{eq:k-beta}, then
\eqref{eq:scaled-feature-envelope}--\eqref{eq:occupancy-sum}.  Finally
decrease $\lambda$ so that both
\[
 C\sqrt H\left\|\sum_ja_jw(\,\cdot-\lambda j)-Nw\right\|_{\ell^1}
 \le1
 \quad\text{and}\quad
 \sum_ne_n^{(\lambda)}\le1.
\]
The first condition controls the $\sqrt H$-amplified collapse error, and the
second controls the nonlinear remainder.  This proves
\eqref{eq:collapse-bound}.

We also record the uniformity.  If \(\widetilde q\) is close to \(q\),
the two linear envelopes in \eqref{eq:scaled-feature-envelope}, including
the factor \(\sqrt H\), differ in \(\ell^1\) by at most
\(C\sqrt H\norm{\widetilde q-q}_M\).  Choose the neighborhood first and
a common norm bound \(N_*\) on it.  One can then choose \(\lambda\) so that
\(C\lambda N_*\le\delta_\kappa\),
\(C\sqrt H\lambda N_*^2\le1\), and the collapse
\eqref{eq:l1-frequency-collapse} for the center datum is as accurate as
required.  These choices prove the last assertion.
Quantitatively, the choices can be ordered as follows.  First choose a
neighborhood radius $\delta>0$ so that
$C\sqrt H\,\delta\le1/4$, and a common bound $N_*$ on that neighborhood.
Second choose a finite set $|j|\le J$ so that the $a_j(q)$ tail contributes
at most $1/(4C\sqrt H)$ to the uniform translation bound above.  Finally
choose one $\lambda>0$ so that
\[
 C\lambda N_*\le\delta_\kappa,\qquad
 C\sqrt H\,\lambda N_*^2\le1/4,
\]
and so that the finitely many translation errors for $|j|\le J$ contribute
at most $1/4$.  For any $\widetilde q$ in the chosen neighborhood, the
difference of its box weights from those of $q$ contributes at most
$C\sqrt H\norm{\widetilde q-q}_M\le1/4$.  Thus this single $\lambda$ works
for the whole neighborhood.
\end{proof}

\section{Continuation of the mild solution}\label{sec:endpoint}

\begin{proof}[Proof of Theorem~\ref{thm:main}]
Let $q_0\in M$ and let $q$ be its maximal mild solution from
Proposition~\ref{prop:local}.  Fix $T>0$, put $N=\|q_0\|_M$, and choose
\begin{equation}\label{eq:H-choice}
 H=C_*(1+T)^{3/2}(1+N).
\end{equation}
Proposition~\ref{prop:frequency-collapse} gives a sufficiently small
$0<\lambda\le1$.  Set
\begin{equation}\label{eq:bootstrap-parameters}
 p_0=S_\lambda q_0,\qquad R=H/\lambda^2,\qquad \tau=T/\lambda^2,
\end{equation}
and let $p$ be the maximal mild solution with datum $p_0$.
By scaling and uniqueness it agrees with
$p(s,x)=\lambda q(\lambda^2s,\lambda x)$ on their common lifespan.
For a fixed constant $C_0$ we have
\[
 k_R(p_0)\le A_H:=C_0(1+H^{1/6}N^{1/3}).
\]
The parameter choice gives
\begin{align*}
 \frac TH(1+2A_H)^2
 &\le C\left(\frac TH+T H^{-2/3}N^{2/3}\right)\\
 &\le C(C_*^{-1}+C_*^{-2/3})\frac{T}{1+T}.
\end{align*}
Choose $C_*$, depending only on the fixed constants of the proof, so large
that this is smaller than $\log2/C$, where $C$ is the flux constant.

For every $0<\sigma<\min(\tau,T_+(p))$, we have
$\sigma/R\le T/H$.  The mild-solution bootstrap in
Proposition~\ref{prop:integrated-shear} therefore yields
\[
 \sup_{0\le s\le\sigma}k_R(p(s))\le2k_R(p_0)\le2A_H.
\]
The zero datum is covered by uniqueness.  Coercivity gives, uniformly on
all these shorter existing intervals,
\[
 \|p(s)\|_M\le C A_H(1+2A_H)^9.
\]
The blow-up alternative excludes $T_+(p)\le\tau$, including equality.
Scaling back continues $q$ through time $T$.  Since $T$ was arbitrary,
the solution exists for all positive times.  Applying the same argument
to $\overline{q_0}$ and using $q(t)=\overline{r(-t)}$ gives negative times.
Uniqueness on overlaps gives one global mild solution.

For completeness, local Lipschitz dependence on a fixed interval follows
directly from local stability.  Fix the reference solution $q$ and $T>0$;
its continuous trajectory has a finite bound $B$ on $[-T,T]$.  Choose a
local contraction length $\delta$ and Lipschitz constant $L\ge1$ for
data in the ball of radius $B+1$.  Subdivide the interval into finitely
many steps of length at most $\delta$ in each time direction.
Choose a neighborhood of $q_0$ so small that its initial differences,
multiplied by $L$ to the number of steps, remain less than one.  Inductively,
the perturbed endpoint data at every subdivision time stay in that ball;
local stability and uniqueness apply on the next step.  Hence
\[
 \|\widetilde q-q\|_{C([-T,T];M)}
 \le C_{T,q_0}\|\widetilde q_0-q_0\|_M
\]
throughout that neighborhood.  No scale choice uniform over an entire
norm ball is required.  The frequency truncations were used only for the
local justification of the microscopic law, before the global argument.
\end{proof}

\end{document}